\documentclass[11pt,a4paper]{amsart}
\usepackage{amsfonts}
\usepackage{}
\usepackage{amsfonts}
\usepackage{amsmath,xypic}
\usepackage{mathrsfs}
\usepackage{amssymb}

\usepackage{CJK,CJKnumb}
\usepackage[CJKbookmarks,colorlinks,
            linkcolor=black,
            anchorcolor=black,
            citecolor=blue]{hyperref}
\usepackage{color}              % 支持彩色
\usepackage{indentfirst}        %首行缩进宏包
\usepackage{latexsym,bm}        % 处理数学公式中和黑斜体的宏包
\usepackage{amsmath,amssymb}    % AMSLaTeX宏包 用来排出更加漂亮的公式
\usepackage{pstricks}
\usepackage{pst-node}
\usepackage{pst-tree}
\usepackage{pst-plot}
\usepackage{pst-text}
\usepackage{graphicx}
\usepackage{cases}
\usepackage{pifont}
\usepackage{txfonts}
\usepackage[all,knot,poly]{xy}

\allowdisplaybreaks[4]  % \eqnarray如果很长，影响分栏、换行和分页
\CJKtilde   %用于解决英文字母和汉字的间距问题。例如：变量~$x$~的值。

\newcommand{\R}{\mathscr{R}}

\newtheorem{corollary}{Corollary}[section]
\newtheorem{theorem}{Theorem}[section]
\newtheorem{lemma}{Lemma}[section]
\newtheorem{remark}{Remark}[section]
\newtheorem{proposition}{Proposition}[section]
\newtheorem{definition}{Definition}[section]
\newtheorem{example}{Example}[section]
\allowdisplaybreaks[4]
                               { \end{spacing}
                               } %第二部分参数结束
\begin{document}
\title[Double-bosonization and Majid's conjecture, (I): Rank-inductions of $ABCD$]{Double-bosonization and Majid's conjecture, (I): Rank-inductions of $ABCD$}
\author[H. Hu]{Hongmei Hu$^{\dagger, 1}$}
\address{$^{\dagger}$School of Mathematical Sciences, University of Science and Technology of China, Jin Zhai Road 96,
Hefei 230026,
PR China}
\email{hmhu0124@126.com}

\author[N. Hu]{Naihong Hu$^{\ast,1}$}
\address{$^1$Department of Mathematics,  Shanghai Key Laboratory of Pure Mathematics and Mathematical Practice,
East China Normal University,
Minhang Campus,
Dong Chuan
Road 500,
Shanghai 200241,
PR China}
\email{nhhu@math.ecnu.edu.cn}

\thanks{N.~H., supported by the NNSFC (Grant No.
 11271131).}

\thanks{$\ast$ corresponding author.}

\date{}
\maketitle

\newcommand*{\abstractb}[3]{ %
                             \begingroup%
                             \leftskip=8mm \rightskip=8mm%设置边距
                             \fontsize{11pt}{\baselineskip}\noindent{\textbf{Abstract} ~}{#1}\\ %摘要
                             {\textbf{Keywords} ~}{#2}\\ %关键词
                             {\textbf{MR(2010) Subject Classification} ~}{#3}\\ %主题分类
                                                          \endgroup
                           }
\begin{abstract}
Majid developed in \cite{majid3} the double-bosonization theory to
construct $U_q(\mathfrak g)$ and expected to generate inductively
not just a line but a tree of quantum groups starting from a node.
In this paper, the authors confirm the Majid's first expectation
(see p. 178 \cite{majid3}) through giving and verifying the full details of the
inductive constructions of $U_q(\mathfrak g)$ for the classical
types, i.e., the $ABCD$ series. Some examples in low ranks are
given to elucidate that any quantum group of classical type can be
constructed from the node corresponding to $U_{q}(\mathfrak{sl}_2)$.
%and describe the relations between this theory and quantum shuffle theory of Rosso in \cite{rosso}.
\end{abstract}
%%%%%%%%%%%%%%%%%%%%%%%%%%%%%%
%\end{format}      %摘要格式结束
%%%%%%%%%%%%%%%%%%%%%%%%%%%%%%%%%%%%%%%%%%%%%%%%%%%%%%%%%%%%%%%%%%%%%%%%%%%%%%%%%%%%%%%%%%%%%
%%%%%%%%%%%%%%%%%%%%%%%%%%%%%         摘要填充         %%%%%%%%%%%%%%%%%%%%%%%%%%%%%%%%%%%%%%
%%%%%%%%%%%%%%%%%%%%%%%%%%%%%%%%%%%%%%%%%%%%%%%%%%%%%%%%%%%%%%%%%%%%%%%%%%%%%%%%%%%%%%%%%%%%%
%%%%%%%%%%%%%%%%%%%%%%%%%%%%%         正文填充         %%%%%%%%%%%%%%%%%%%%%%%%%%%%%%%%%%%%%%
%%%%%%%%%%%%%%%%%%%%%%%%%%%%%%%%%%%%%%%%%%%%%%%%%%%%%%%%%%%%%%%%%%%%%%%%%%%%%%%%%%%%%%%%%%%%%
%\begin{format}%正文格式开始
%%%%%%%%%%%%%%%%%%%%%%%%%%%%%         开始工作         %%%%%%%%%%%%%%%%%%%%%%%%%%%%%%%%%%%%%%%%%%%%%%%%%%%%%%%%%%%%%

\section{Introduction}

The invention of quantum groups is one of the outstanding
achievements of mathematical physics and mathematics in the late
twentieth century, which arose in the work of L. D. Faddeev with his
school for solving integrable models in $1{+}1$ dimension by the
quantum inverse scattering method. In the history of development, a
major event was the discovery of quantized universal enveloping
algebras $U_q(\mathfrak g)$ over the complex field by V. G. Drinfeld
\cite{dri} and M. Jimbo \cite{jimbo} independently around 1985. A
striking feature of quantum group theory is the close connections
with many branches of mathematics and physics, such as Lie groups,
Lie algebras and their representations, representation theory of
Hecke algebras, link invariant theory, conformal field theory, and
so on. This attracts many mathematicians to find some better way in
a suitable frame to understand the structure of quantum groups
defined initially by generators and relations. For instance, the
first we must mention is the famous FRT-construction \cite{FRT1} of
$U_q(\mathfrak g)$ for the classical types based on the $R$-matrices
of the vector representations associated to the classical simple Lie
algebras $\mathfrak g$, which is a natural analogue of the matrix
realization of the classical Lie algebras. It is a belief that the
commutation relations in quantum groups expressed by means of the
$R$-matrices are of fundamental importance for the theory.
Afterwards, Majid rediscovered the so-called ``Radford-Majid
bosonization" arising from a framework of a braided category over a Hopf
algebra (see \cite{rad}, \cite{rt}, \cite{majid1}, \cite{majid2},
\cite{majid9}), and then, based on the FRT-construction and
extending the Drinfeld double to the generalized quantum double
associated to a dual pair of Hopf algebras equipped with a (not
necessarily non-degenerate) pairing, Majid \cite{majid3} in 1996
developed the double-bosonization theory to give a direct
construction of $U_q(\mathfrak g)$ through viewing the Lusztig's
algebra $\mathfrak f$ (\cite{lus2}) as a braided group in a special braided
category  $\mathfrak{M}_H$ (or ${}_H\mathfrak{M}$), where
($H=kQ=k[K_i^{\pm1}]$, $A=kQ^\vee$) is a weakly quasitriangular dual
pair (see p. 169 of \cite{majid3}), where $Q$ (resp. $Q^\vee$) is a (dual) root lattice of
$\mathfrak g$. An analogous construction in spirit in the Yetter-Drinfeld category of a Hopf algebra was given
independently by Sommerh\"auser in \cite{somm}. 

Besides the constructions above, in the early 90s, Ringel \cite{ringel} realized the positive part of
quantum groups by quiver representations and Hall algebras,
which inspired Lusztig's canonical bases theory \cite{lus1,lus2}.
Rosso \cite{rosso} also realized the positive part of
$U_q(\mathfrak g)$ by introducing the quantum shuffle algebra in a
braided category, and gave a recipe of axiomatic inductive construction
based on the quantum ``shuffle" operation defined on the (braided)
tensor algebra of a braided vector space which is sometimes a bit
inconvenient for making practical calculations, in contrast with the
Majid's double-bosonization \cite{majid3}. Bridgeland \cite{Bri}
recently realized the entire quantum groups of type $ADE$ using the
Ringel-Hall algebras. From another way, Fang-Rosso \cite{fang}
realized the whole quantum groups by means of a new theory on the
quantum quasi-symmetric algebras due to Jian-Rosso \cite{jr}. More
general, the whole axiomatic description for the multi-parameter
quantum groups as well as the Hopf $2$-cocycle deformation under the
machinery of multi-parameter quasi-symmetric algebras have been
achieved in Hu-Li-Rosso \cite{hlr}.

In this paper, let us focus on the double-bosonizaiton theory in \cite{majid3}.%, that is,
%to give a complete description of the inductive or type-crossing construction for $U_q(\mathfrak g)$ (where $\mathfrak g$ is of classical type $A, B, C, D$, respectively).

Associated to any mutually dual braided groups $B^{\star}, B$
covariant under a background quasitriangular Hopf algebra H, there
is a new quantum group on the tensor space $B^{\star}\otimes
H\otimes B$ by double-bosonization in \cite{majid3}, consisting of
$H$ extended by $B$ as additional `positive roots' and its dual
$B^{\star}$ as additional `negative roots'. The construction is more
powerful than the quantum double since one can reach directly to the
$U_q(\mathfrak g)$ (the quantum double is a bit big and one has to
make quotient). Specially, Majid viewed $U_{q}(\mathfrak n^{\pm})$ as the
mutually dual braided groups in braided category of right ``Cartan
subalgebra" $H$-modules, then recovered $U_q(\mathfrak g)$ by his
double-bosonization theory. Also, based on two
examples in low ranks given in \cite{majid3, majid8}, Majid claimed that many new
quantum groups, as well as the inductive (by rank) construction of
$U_q(\mathfrak g)$ can be obtained in principle by this theory \cite{majid3} (we call it the Majid's expectation),
since he imaged his double-bosonization framework allows to
generate a tree of quantum groups and at each node of the tree,
there are many choices to adjoin a pair of braided groups covariant under
the corresponding quantum group at that node. Actually, it is a combinatorial and representation-theoretical
challenge to elaborate the full tree structure, which brings us the first motivation. We want to describe in details
the inductive construction of quantum groups
$U_q(\mathfrak g)$ for all complex semisimple Lie algebras
$\mathfrak g$, which just elaborates some main branches of the tree. We will
limit ourself to consider the classical $ABCD$ series in this paper,
which is organized as follows.

In section 2, we recall some basic facts about the FRT-construction
and the Majid's double-bosonization construction. In section 3, we
analyse explicitly the procedure of the inductive construction of
$U_{q}({\mathfrak {sl}}_{n})$, and find some tips in low rank cases. That is,
although the entire $FRT$-matrix $m^{\pm}$ is hard to know from the
vector representation, we only need to know those diagonal and minor
diagonal elements. This inspires us how to proceed the general
rank-inductive construction from $U_{q}({\mathfrak {sl}}_{n})$ to $U_{q}({\mathfrak {sl}}_{n+1})$. At
the node diagram of $U_q({\mathfrak {sl}}_n)$, we can choose a pair of braided groups
corresponding to the standard $R$-matrix arising from the vector representation. Similarly, we can do the same thing for
the $BCD$ series. These demonstrate that the Majid's
double-bosonization framework does allow to generate four
branching-lines of nodes diagram of quantum groups. Furthermore, in
the last section, we will give some examples to show how to grow the
tree (with $ABCD$-branches) of nodes diagram of quantum groups out
of the same `root' node at type $A$ via type-crossing construction starting
from type $A_r$ for $r=1,2,3$.

\section{FRT-construction and Majid's double-bosonization construction}
In this paper,
let $k$ be the complex field,
 $\mathfrak g$ a finite-dimensional complex simple Lie algebra with simple roots $\alpha_{i}$.
Let $\lambda_{i}$ be the fundamental weight corresponding to
$\alpha_i$. Cartan matrix of $\mathfrak g$ is $(a_{ij})$, where
$a_{ij}=\frac{2(\alpha_{i},\alpha_{j})}{(\alpha_{i},\alpha_{i})}$,
and $d_{i}=\frac{(\alpha_{i},\alpha_{i})}{2}$. Let $(H,\R)$ be a
quasitriangular Hopf algebra, where $\R$ is the universal
$R$-matrix, $\R=\R^{(1)}\otimes \R^{(2)}$, $\R_{21}=\R^{(2)}\otimes
\R^{(1)}$. Denote by $\Delta, \,\eta, \,\epsilon$, $S$ its
coproduct, counit, unit, its antipode, respectively. We shall use
Sweedler's notation: for $h \in H$, $\Delta(h)=h_{1}\otimes h_{2}$.
Write $H^{op} \ (H^{cop})$ the opposite (co)algebra structure of
$H$, respectively. Let $\mathfrak{M}_{H} \ ({}_{H}\mathfrak{M}$) be
the braided category consisting of right (left) $H$-modules,
respectively. If there exists a coquasitriangular Hopf algebra $A$
such that $(H, A)$ is a weakly quasitriangular dual pair, then
$\mathfrak{M}_{H} \ ({}_{H}\mathfrak{M}$) is equivalent to the
braided category ${}^{A}\mathfrak{M}\ (\mathfrak{M}^{A})$ consisting
of left (right) $A$-comodules, respectively. For the detailed
description of these theories, we left to the readers to refer to
Drinfeld's and Majid's papers \cite{dri}, \cite{majid4},
\cite{majid5}, and so on. By a braided group, following
\cite{majid4}, we mean a braided bialgebra or Hopf algebra in some
braided category. In order to distinguish from the ordinary Hopf
algebras, denote by $\underline{\Delta},\underline{S}$ its coproduct
and antipode, respectively. An invertible matrix solution of the
quantum Yang-Baxter equation (QYBE)
$R_{12}R_{13}R_{23}=R_{23}R_{13}R_{12}$ is called a $R$-matrix. For
later use, one needs a certain extension $U_q^{ext}(\mathfrak g)$ of
$U_q(\mathfrak g)$. This is constructed by adjoining formally
certain products of the elements $K_{i}^{\pm\frac{1}{n}}$,
$i=1,\cdots,n-1$ to $U_{q}({\mathfrak {sl}}_{n})$, the elements
$K_{n}^{\pm\frac{1}{2}}$ to $U_{q}({\mathfrak {sp}}_{2n})$, and
$K_{n-1}^{\pm\frac{1}{2}}K_{n}^{\pm\frac{1}{2}}$ and
$K_{n-1}^{\pm\frac{1}{2}}K_{n}^{\mp\frac{1}{2}}$ to
$U_{q}({\mathfrak {so}}_{2n})$. The precise definitions of them can be found in
\cite{klim}. Denote by $T_{V}$ an irreducible representation of
$U_{q}(\mathfrak{g})$, where $V$ is the corresponding module with a
basis $\{\,x_i\,\}$.

\subsection{$FRT$-construction}
One can obtain an invertible (basic) $R$-matrix from the quasitriangular Hopf algebra $U_{q}(\mathfrak{g})$ and its (vector) representation.
Conversely,
starting from an invertible (basic) $R$-matrix, if does there exist a quasitriangular Hopf algebra to recover such a $R$-matrix through a suitable representation? First of all, by Faddeev-Reshetikhin-Takhtajan \cite{FRT1},
the following fact is basic and well-known.
\begin{definition}
Given an invertible matrix solution $R$ of the QYBE,
there is a bialgebra $A(R)$,
named the $FRT$-bialgebra,
which is generated by $1$ and $t^{i}_{j}$, for $1\le i, j\le n$,
with the relations
$RT_{1}T_{2}=T_{2}T_{1}R, \
\Delta(T)=T\otimes T$
and
$\epsilon(T)=I$ using standard notation in \cite{FRT1},
where the matrix $T=(t^{i}_{j})$,
$T_{1}=T\otimes I$,
$T_{2}=I\otimes T$.
\end{definition}

Observe that $A(R)$ is a coquasitriangular bialgebra with
$\R:\, A(R)\otimes A(R) \longrightarrow k$ such that $\R(t^{i}_{j}\otimes t^{k}_{l})=R^{ik}_{jl}$.
Here $R^{ik}_{jl}$ denotes the entry at row $(ik)$ and column $(jl)$ in matrix $R$.
Secondly, in the dual space $A(R)^{\ast}=Hom(A(R),k)$,
$\Delta$ of $A(R)$ induces the multiplication of $A(R)^{\ast}$. In \cite{FRT1},
$U_{R}$ is defined to be the subalgebra of $A(R)^{\ast}$ generated by $L^{\pm}=(l^{\pm}_{ij})$,
with relations
$$(PRP)L_{1}^{\pm}L_{2}^{\pm}=L_{2}^{\pm}L_{1}^{\pm}(PRP), \qquad
(PRP)L_{1}^{+}L_{2}^{-}=L_{2}^{-}L_{1}^{+}(PRP),
$$
where $l^{\pm}_{ij}$ is defined by
$
(l^{+}_{ij},t^{k}_{l})=R^{ki}_{lj},
$
$
(l^{-}_{ij},t^{k}_{l})=(R^{-1})^{ik}_{jl}.
$

The algebra $U_{R}$ is a bialgebra with coproduct
$\Delta(L^{\pm})=L^{\pm}\otimes L^{\pm}$ and counit
$\varepsilon(l_{ij}^{\pm})=\delta_{ij}$. Specially, when $R$ is the
classical $R$-matrix, bialgebra $A(R)$ has a quotient
coquasitriangular Hopf algebra, denoted by $Fun(G_{q})$ or
$\mathcal{O}_{q}(G)$, and $U_{R}$ also has a corresponding quotient
quasitriangular Hopf algebra, which is isomorphic to the extended
quantized enveloping algebra $U_{q}^{ext}(\mathfrak{g})$. Moreover,
there exists a (non-degenerate) dual pairing $\langle ,\, \rangle$
between $\mathcal{O}_{q}(G)$ and $U_{q}^{ext}(\mathfrak{g})$. The
way of getting the resulting quasitriangular algebras
$U_{q}^{ext}(\mathfrak{g})$ is the so-called FRT-construction of the
quantized enveloping algebras (for the classical types).

Motivated by the work of \cite{FRT1},
Majid built the theory of the weakly quasitriangular dual pairings associated with $R$-matrices \cite{majid3} in a more general context.
\begin{remark}\label{note}
The $R$-matrices used in Majid's papers \cite{majid1, majid3} are a bit different from the standard ones as in \cite{FRT1}, which are the conjugations $P\circ\cdot\circ P$ of the ordinary $R$-matrices by the permutation matrix $P: P(u\otimes v)=v\otimes u$.
It can be checked directly that $A(P\circ R\circ P)=A(R)^{\text{op}}$, where $(P\circ R\circ P)^{ij}_{kl}=R^{ji}_{lk}$.
Note that we will use Majid's notation for $R$-matrices as in \cite{majid3} in the remaining sections of this paper.
\end{remark}

\subsection{Majid's double-bosonization}
Majid \cite{majid3} proposed the concept of a weakly quasitriangular dual pair via his insight on more examples on matched pairs of bialgebras or Hopf algebras in \cite{majid7}.
This allowed him to establish a theory of double-bosonization in a broad framework that generalized the FRT's construction which was limited to the classical types.
\begin{definition}
Let $(H,A)$ be a pair of Hopf algebras equipped with a dual pairing $\langle ,\,\rangle$
and convolution-invertible algebra\,/\,anti-coalgebra maps $\R,\bar{\R}: A \rightarrow H$ obeying
$$\langle\bar{\R}(a),b\rangle=\langle\R^{-1}(b),a\rangle,\quad
\partial^{R}h=\R \ast(\partial^{L}h)\ast\R^{-1},\quad
\partial^{R}h=\bar{\R}\ast(\partial^{L}h)\ast\bar{\R}^{-1}
$$
for
$a,b\in A, \ h\in H$.
Here $\ast$ is the convolution product in $hom(A,H)$ and
$(\partial^{L}h)(a)=\langle h_{(1)},a\rangle h_{(2)},
$
$
(\partial^{R}h)(a)=h_{(1)}\langle h_{(2)},a\rangle$
are left, right ``differentiation operators" regarded as maps $A \rightarrow H$ for fixed $h$.
\end{definition}
Let $C, B$ be a pair of braided groups in $\mathfrak{M}_{H}$, which are called
dually paired if there is an intertwiner
$ev: C\otimes B \longrightarrow k$ such that
$
\text{ev}(cd,b)=\text{ev}(d,b_{\underline{(1)}})ev(c,b_{\underline{(2)}}),$
$
\text{ev}(c,ab)=\text{ev}(c_{\underline{(2)}},a)ev(c_{\underline{(1)}},b),$
$
\forall a,b\in B,c,d\in C.
$
Then $C^{op/cop}$ (with opposite product and coproduct) is a Hopf algebra in $_{H}\mathfrak{M}$,
which is dual to $B$ in the sense of an ordinary dual pairing $\langle~,~\rangle$ with
$H$-bicovariant: $\langle h\rhd c,b\rangle=\langle c, b\lhd h\rangle$ for all $h\in H$.
Let $\overline{C}=(C^{op/cop})^{\underline{cop}}$,
then $\overline{C}$ is a braided group in $_{\overline{H}}\mathfrak{M}$,
where $\overline{H}$ is $(H,\R_{21}^{-1})$.
With these,
Majid gave the following double-bosonization theorem.
\begin{theorem}\label{ml1} $(${\rm\bf Majid}$)$
On the tensor space $\bar{C}\otimes H \otimes B$,
there is a unique Hopf algebra structure $U=U(\bar{C},H,B)$ such that
$H\ltimes B$ (bosonization) and $\bar{C}\rtimes H$ (bosonization) are sub-Hopf algebras by the canonical inclusions,
with cross relation
$$
bh=h_{(1)}(b\lhd h_{(2)}),\quad
ch=h_{(2)}(c\lhd h_{(1)}),
\eqno{(C1)}$$
here
$b\in B,
c\in C,
h\in H$.
% and
%$$
%\left.
%\begin{array}{rl}
%bc=&(\R_{1}^{(2)}\rhd c_{\overline{(2)}})
%\R_{2}^{(2)}\R_{1}^{-(1)}(b_{\underline{(2)}}\lhd\R_{2}^{-(1)})\,
%\langle\R_{1}^{(1)}\rhd c_{\overline{(1)}},b_{\underline{(1)}}\lhd
%\R_{2}^{(1)}\rangle\,\cdot\\
%&\hskip2cm\cdot\,\langle\R_{1}^{-(2)}\rhd \overline{S}c_{\overline{(3)}},b_{\underline{(3)}}\lhd \R_{2}^{-(2)}\rangle,
%\end{array}
%\right.
%$$
%$\forall\, b\in B, \, c\in\overline{C}$.
%Here $\R_{1},\, \R_{2}$
%are distinct copies of the quasitriangular structure $\R$ of $H$.
If there exists a coquasitriangular Hopf algebra $A$ such that
$(H,A)$ is a weakly quasitriangular dual pair, and $b, c$ are
primitive elements, then some relations simplify to
$$
[b,c]=\R(b^{\overline{(1)}})\langle c,b^{\overline{(2)}}\rangle-
\langle c^{\overline{(1)}},b\rangle\bar{\R}(c^{\overline{(2)}});\eqno{(C2)}
$$
$$
\Delta b=b^{\overline{(2)}}\otimes \R(b^{\overline{(1)}})+1\otimes b,\quad
\Delta c=c\otimes 1+\bar{\R}(c^{\overline{(2)}})
\otimes c^{\overline{(1)}}.\eqno{(C3)}
$$
\end{theorem}
Let $\widetilde{U(R)}$ be the double cross product bialgebra of
$A(R)^{\text{op}}$ in \cite{majid7} generated by $m^{\pm}$ with the
 bialgebra structure given by
$$
Rm^{\pm}_{1}m^{\pm}_{2}=m^{\pm}_{2}m^{\pm}_{1}R,\quad
Rm_{1}^{+}m_{2}^{-}=m_{2}^{-}m_{1}^{+}R,\quad
\Delta((m^{\pm})^{i}_{j})=(m^{\pm})_{j}^{a}\otimes
(m^{\pm})_{a}^{i},\quad
\epsilon((m^{\pm})^{i}_{j})=\delta_{ij},\eqno{(C4)}
$$
where $(m^{\pm})^{i}_{j}$ are the $FRT$-generators, $m^{\pm}$ are
called $FRT$-matrices.
\begin{proposition}\label{weakly}
$(1)$
$(\widetilde{U(R)},A(R))$ is a weakly quasitriangular dual pair with
$$\langle(m^{+})^{i}_{j},t^{k}_{l}\rangle= R^{ik}_{jl},
\quad
\langle (m^{-})^{i}_{j},t^{k}_{l}\rangle =(R^{-1})^{ki}_{lj},
\quad
\R(T)=m^{+},\quad
\bar{\R}(T)=m^{-}.$$
Specially,
when $R$ is the classical $R$-matrix,
then the weakly quasitriangular dual pair can be descended to a mutually dual pair of quotient Hopf algebras $(U_{q}^{ext}(\mathfrak g),\mathcal{O}_q(G))$  (where $G$ is one of connected and simply connected algebraic groups of classical types) with the correct modification \cite{majid3}:
$$
\langle (m^{+})^{i}_{j},t^{k}_{l}\rangle=\lambda R^{ik}_{jl}, \quad
\langle (m^{-})^{i}_{j},t^{k}_{l}\rangle=\lambda^{-1} (R^{-1}){}^{ki}_{lj}.
$$
Such $\lambda$ is called a {\it quantum group normalization constant}.

$(2)$ Suppose that $R'$ is another matrix such that (\text{i}) \
$R_{12}R_{13}R'_{23}=R'_{23}R_{13}R_{12}$, \ $R_{23}R_{13}R'_{12}$
$=R'_{12}R_{13}R_{23}$, \ (\text{ii}) \ $(PR+1)(PR'-1)=0$, \
(\text{iii}) $R_{21}R'_{12}=R'_{21}R_{12}$, where $P$ is the
permutation matrix with entries
$P^{ij}_{kl}=\delta_{il}\delta_{jk}$. Then a braided-vector algebra
$V(R', R)$ generated by generators $1$, $\{e^{i}~|~i=1,\cdots,n\}$,
 and relations
$e^{i}e^{j}=\sum\limits_{a,b}{R'}^{ji}_{ab}e^{a}e^{b}$
forms a braided group with
$\underline{\Delta}(e^{i})=e^{i}\otimes 1+1\otimes e^{i},
\underline{\epsilon}(e^{i})=0,
\underline{S}(e^{i})=-e^{i},
\Psi(e^{i}\otimes e^{j})=\sum\limits_{a,b}R^{ji}_{ab}e^{a}\otimes e^{b}$
in braided category ${}^{A(R)}\mathfrak{M}$.
Under duality $\langle f_{j}, e^{i}\rangle=\delta_{ij}$, a
braided-covector algebra $V^{\vee}(R', R_{21}^{-1})$ generated by 1 and $\{f_{j}~|~j=1,\cdots,n\}$,
and relations
$f_{i}f_{j}=\sum\limits_{a,b}f_{b}f_{a}{R'}_{ab}^{ij}$
forms another braided group with
$\underline{\Delta}(f_{i})=f_{i}\otimes 1+1\otimes f_{i}$,
$\underline{\epsilon}(f_{i})=0$,
$\underline{S}(f_{i})=-f_{i}$,
$\Psi(f_{i}\otimes f_{j})=\sum\limits_{a,b}f_{b}\otimes f_{a}R^{a}_{i}{}^{b}_{j}$
in braided category $\mathfrak{M}^{A(R)}$.
\end{proposition}

\begin{remark}\label{op}
In fact, by Remark \ref{note}, $\widetilde{U(R)}$ determined
uniquely by relations $(C4)$ with the $R$-matrix in \cite{majid3} is
the opposite of $U_{R}$ constructed by $FRT$-approach. So the
quantized enveloping algebra $U_{q}(\mathfrak{g})$ recovered by
Majid (Proposition 4.3 in \cite{majid3}) satisfies $
E_{i}K_{j}=q_{i}^{a_{ij}}K_{j}E_{i}, $ $
F_{i}K_{j}=q_{i}^{-a_{ij}}K_{j}F_{i}, $ $
[E_{i},F_{j}]=\delta_{ij}\frac{K_{i}-K_{i}^{-1}}{q-q^{-1}}, $ where
$q_{i}=q^{d_{i}}$, which is the opposite of the ordinary
$U_{q}(\mathfrak{g})$. By abuse of notation, we still use $U_{q}(\mathfrak{g})$
instead of $U_{q}(\mathfrak{g})^{op}$ in the
remaining sections of this paper.
\end{remark}

With the new $\lambda R$-matrix in {\bf Proposition \ref{weakly}},
in order that $V(R^{\prime},R),V^{\vee}(R^{\prime},R_{21}^{-1})$ are still braided groups ,
then $(U_{q}^{ext}(\mathfrak g),\mathcal{O}_q(G))$ must be centrally extended to the pair
$$\Bigl(\widetilde{U_{q}^{ext}(\mathfrak g)}=U_{q}(\mathfrak g)\otimes k[c,c^{-1}],\widetilde{\mathcal{O}_q(G)}=\mathcal{O}_q(G)\otimes k[g,g^{-1}]\Bigr)$$
with action $e^{i}\lhd c=\lambda e^{i}$,
$f_{i}\lhd c=\lambda f_{i}$, $\langle c,g\rangle=\lambda$.
By Theorem \ref{ml1},
we have the following
\begin{corollary}\label{cor1}
$U=U(V^{\vee}(R^{\prime},R_{21}^{-1}),\widetilde{U_{q}^{ext}(\mathfrak g)},V(R^{\prime},R))$ is a new quantum group with the cross relations:
$e^{i}(m^{+})^{j}_{k}=\lambda R^{ji}_{ab}(m^{+})^{a}_{k}e^{b}$,
$(m^{-})^{i}_{j}e^{k}=\lambda R^{ki}_{ab}e^{a}(m^{-})^{b}_{j}$,
$(m^{+})^{i}_{j}f_{k}=\lambda f_{b}(m^{+})^{i}_{a}R^{ab}_{jk}$,
$f_{i}(m^{-})^{j}_{k}=\lambda (m^{-})^{j}_{b}f_{a}R^{ab}_{ik}$,
$cf_{i}=\lambda f_{i}c$,
$e^{i}c=\lambda ce^{i}$,
$[c,m^{\pm}]=0,$
$[e^{i},f_{j}]=\delta_{ij}\frac{(m^{+})^{i}_{j}c^{-1}-c(m^{-})^{i}_{j}}{q_{\ast}-q_{\ast}^{-1}};$
and the coproduct:
$\Delta c=c\otimes c,$
$\Delta e^{i}=e^{a}\otimes (m^{+})^{i}_{a}c^{-1}+1\otimes e^{i},$
$\Delta f_{i}=f_{i}\otimes 1+c(m^{-})^{a}_{i}\otimes f_{a},$
$\epsilon e^{i}=\epsilon f_{i}=0$,
where one can normalize $e^{i}$ such that the factor $q_{\ast}-q_{\ast}^{-1}$ satisfies the situation you need.
\end{corollary}
So the Majid's double-bosonization construction can lead to new quantum
groups. After giving an example of obtaining $U_{q}({\mathfrak {sl}}_{3})$ from $
U_{q}^{ext}({\mathfrak {sl}}_{2})$ in \cite{majid3}, Majid expected that the novel resulting quantum group
is the quantum group of higher-one rank in the classical $ABCD$'s
series. In the current paper, we will solve such Majid's expectation in the classical types that has been
an open question aimed by Majid (\cite{majid3}) since the mid of 90's. In the next section, we will
give full details of the rank-inductive construction in the $ABCD$
series.

\section{Rank-inductive construction of quantum groups for classical types}
In order to explore the structure of the resulting quantum group in {\bf Corollary \ref{cor1}},
we need to know how to get the explicit form of the $FRT$-matrix $m^{\pm}$,
which can be obtained by the following lemma.
\begin{lemma}\label{lem1}
Corresponding to the invertible matrix $R$ obeying the QYBE, we have
$S(l_{ij}^{\pm})=(m^{\pm})^{i}_{j}$ in the quotient Hopf algebra,
$S$ is the corresponding antipode.
\end{lemma}
\begin{proof}
If there exists a quotient Hopf algebra of $\widetilde{U(R)}$,
according to {\bf Remark \ref{op}}, we obtain the following
relations in this quotient Hopf algebra:
$RL_{2}^{\pm}L_{1}^{\pm}=L_{1}^{\pm}L_{2}^{\pm}R,$ $
RL_{2}^{-}L_{1}^{+}=L_{1}^{+}L_{2}^{-}R, $ and $
\Delta(L^{\pm})=L^{\pm}\otimes L^{\pm}$. We can describe these
relations in view of the entries in the matrix $R$ and $L^{\pm}$.
For example, for any fixed $i,j,k,l$, we have $
(RL_{2}^{\pm}L_{1}^{\pm})^{ij}_{kl}=(L_{1}^{\pm}L_{2}^{\pm}R)^{ij}_{kl}
$, then we obtain the following equality on the left hand side
$$\left.
\begin{array}{rl}
(RL_{2}^{\pm}L_{1}^{\pm})^{ij}_{kl}
&=R^{ij}_{ab}(I\otimes L^{\pm})^{ab}_{mn}(L^{\pm}\otimes I)^{mn}_{kl}
=R^{ij}_{ab}\delta_{am}(l^{\pm})_{bn}(l^{\pm})_{mk}\delta_{nl}\\
&=R^{ij}_{mb}(l^{\pm})_{bl}(l^{\pm})_{mk}
=R^{ij}_{mn}(l^{\pm})_{nl}(l^{\pm})_{mk},
\end{array}
\right.$$
and the equality by the right hand side
$$(L_{1}^{\pm}L_{2}^{\pm}R)^{ij}_{kl}
=(L^{\pm}\otimes I)^{ij}_{ab}(I\otimes L^{\pm})^{ab}_{mn}R^{mn}_{kl}
=(l^{\pm})_{ia}\delta_{jb}\delta_{am}(l^{\pm})_{bn}R^{mn}_{kl}
=(l^{\pm})_{im}(l^{\pm})_{jn}R^{mn}_{kl}.
$$
So
$
R^{ij}_{mn}(l^{\pm})_{nl}(l^{\pm})_{mk}=(l^{\pm})_{im}(l^{\pm})_{jn}R^{mn}_{kl}.
$
Taking the antipode $S$ on both sides,
we obtain
$$
R^{ij}_{mn}S((l^{\pm})_{mk})S((l^{\pm})_{nl})=S((l^{\pm})_{jn})S((l^{\pm})_{im})R^{mn}_{kl}.
$$
Under the notation $(m^{\pm})^{i}_{j}=S(l_{ij}^{\pm})$,
we have
$$
R^{ij}_{mn}(m^{\pm})^{m}_{k}(m^{\pm})^{n}_{l})=(m^{\pm})^{j}_{n})(m^{\pm})^{i}_{m})R^{mn}_{kl}, \ \text{i.e.}, \
Rm_{1}^{\pm}m_{2}^{\pm}=m_{2}^{\pm}m_{1}^{\pm}R.
$$
We also obtain $Rm_{1}^{+}m_{2}^{-}=m_{2}^{-}m_{1}^{+}R$ by a similar argument.
On the other hand,
\begin{equation*}
\begin{split}
\Delta((m^{\pm})^{i}_{j})&
=\Delta(S(l^{\pm}_{ij}))
=P\circ(S\otimes S)\Delta(l^{\pm}_{ij})
=P\circ(S\otimes S)(l^{\pm}_{ia}\otimes l^{\pm}_{aj})\\
&=P((m^{\pm})^{i} _{a}\otimes (m^{\pm})^{a}_{j})
=(m^{\pm})^{a}_{j}\otimes (m^{\pm})^{i}_{a}.
\end{split}
\end{equation*}
So, the generators $(m^{\pm})^{i}_{j}$ satisfy relation $(C4)$.
This completes the proof.
\end{proof}

\subsection{Inductive construction of $U_{q}({\mathfrak {sl}}_{n})$}
Since $m^{\pm}$ and $R$-matrix will become larger and larger with the growing of rank,
it is a challenge to describe explicitly  the procedure of general inductive construction.
We want to find some tips on some concrete examples.
Majid already described explicitly the case of $U(V^{\vee}(R^{\prime},R_{21}^{-1}),\widetilde{U_{q}^{ext}({\mathfrak {sl}}_{2})},V(R^{\prime},R))\simeq U_{q}({\mathfrak {sl}}_{3})$ in
\cite{majid3}. First of all, in order to capture much more hints coming from the theory of Majid's double-bosonization construction, in what follows, we will describe in detail an example of rank $2$ case: $U(V^{\vee}(R^{\prime},R_{21}^{-1}),\widetilde{U_{q}^{ext}({\mathfrak {sl}}_{3})},V(R^{\prime},R))\simeq U_{q}({\mathfrak {sl}}_4)$.
\begin{example}
Let us start with $R$-matrix datum
$$
R=
\left(
\begin{array}{ccccccccc}
q^{2}&~0&~0&~0&~0&~0&~0&~0&~0\\
0&~q&~0&~q^{2}-1&~0&~0&~0&~0&~0\\
0&~0&~q&~0&~0&~0&~q^{2}-1&~0&~0\\
0&~0&~0&~q&~0&~0&~0&~0&~0\\
0&~0&~0&~0&~q^{2}&~0&~0&~0&~0\\
0&~0&~0&~0&~0&~q&~0&~q^{2}-1&~0\\
0&~0&~0&~0&~0&~0&~q&~0&~0\\
0&~0&~0&~0&~0&~0&~0&~q&~0\\
0&~0&~0&~0&~0&~0&~0&~0&~q^{2}
\end{array}
\right).
$$
Take $R^{\prime}=q^{-2}R$, then $R, R^{\prime}$ give braided groups
$V^{\vee}(R^{\prime},R_{21}^{-1})=k\,\langle f_{i}\mid i=1, 2,
3\rangle$ and $V(R^{\prime},R)=k\,\langle e^{i}\mid i=1, 2,
3\rangle$. Identify $e^{3},f_{3},(m^{+})^{3}_{3}c^{-1}$ with the
additional simple root vectors $E_{3}, F_{3}$ and group-like element
$K_{3}$, then the resulting quantum group
$(V^{\vee}(R^{\prime},R_{21}^{-1}),\widetilde{U_{q}^{ext}({\mathfrak {sl}}_{3})}$,
$V(R^{\prime},R))$ is exactly the quantum group $U_{q}({\mathfrak {sl}}_{4})$ with
$K_{i}^{\frac{1}{3}}, \ i=1, 2$ adjoined.
\end{example}
\begin{proof}
The quantum group normalization constant for $R$ needed for the weakly quasitriangular structure on $U_{q}^{ext}({\mathfrak {sl}}_{3})$ is  $\lambda=q^{-\frac{4}{3}}$ obtained by the facts in \cite{FRT1}.
Moreover,
the $m^{\pm}$-matrix corresponding to the vector representation can be obtained by Lemma \ref{lem1}.
$$
\left.
\begin{array}{rcl}
m^{+}=
\left(
\begin{array}{ccc}
K^{\frac{2}{3}}_{1}K^{\frac{1}{3}}_{2}
&~~(q-q^{-1})
E_{1}K^{-\frac{1}{3}}_{1}K^{\frac{1}{3}}_{2}
&~~q^{-1}(q-q^{-1})
[E_1, E_2]_{q^{-1}}K^{-\frac{1}{3}}_{1}K^{-\frac{2}{3}}_{2}\\
0&K^{-\frac{1}{3}}_{1}K^{\frac{1}{3}}_{2}
&(q-q^{-1})
E_{2}K^{-\frac{1}{3}}_{1}K^{-\frac{2}{3}}_{2}\\
0&0&K^{-\frac{1}{3}}_{1}K^{-\frac{2}{3}}_{2}
\end{array}
\right)
\end{array}
\right.,
$$
$$
\left.
\begin{array}{rcl}
m^{-}=
\left(
\begin{array}{ccc}
K^{-\frac{2}{3}}_{1}K^{-\frac{1}{3}}_{2}&0&0\\
(q-q^{-1})
%q^{-\frac{1}{6}}
K^{\frac{1}{3}}_{1}K^{-\frac{1}{3}}_{2}F_{1}
&K^{\frac{1}{3}}_{1}K^{-\frac{1}{3}}_{2}&0\\
q(q-q^{-1})
K^{\frac{1}{3}}_{1}K^{\frac{2}{3}}_{2}[F_2, F_1]_q
&~~(q-q^{-1})
K^{\frac{1}{3}}_{1}K^{\frac{2}{3}}_{2}F_{2}
&~~K^{\frac{1}{3}}_{1}K^{\frac{2}{3}}_{2}
\end{array}
\right)
\end{array}
\right.,
$$
where $[E_1, E_2]_{q^{-1}}=E_1E_2-q^{-1}E_2E_1$, $[F_2, F_1]_q=F_2F_1-qF_1F_2$.
By Corollary \ref{cor1}, we get
$$
[e^{3},f_{3}]=\frac{(m^{+})^{3}_{3}c^{-1}-c(m^{-})^{3}_{3}}{q-q^{-1}}, \Longrightarrow [E_{3},F_{3}]=\frac{K_{3}-K_{3}^{-1}}{q-q^{-1}}.
$$
$
E_{3}K_{3}=e^{3}(m^{+})^{3}_{3}c^{-1}=\lambda R^{33}_{ab}(m^{+})^{a}_{3}e^{b}c^{-1}
=R^{33}_{ab}(m^{+})^{a}_{3}c^{-1}e^{b}= R^{33}_{33}(m^{+})^{3}_{3}c^{-1}e^{3}=q^{2}K_{3}E_{3}.
$
From the expression of $(m^{+})^{i}_{i}$,
we have
$
(m^{+})^{2}_{2}K_{1}=(m^{+})^{1}_{1},(m^{+})^{3}_{3}K_{2}=(m^{+})^{2}_{2}.
$
Associating with the cross relation
$
e^{3}(m^{+})^{i}_{i}=\lambda R^{i3}_{ab}(m^{+})^{a}_{i}e^{b},
$
we obtain
$$
\left.
\begin{array}{l}
e^{3}(m^{+})^{1}_{1}=q^{-\frac{1}{3}} (m^{+})^{1}_{1}e^{3},\\
e^{3}(m^{+})^{2}_{2}=q^{-\frac{1}{3}} (m^{+})^{2}_{2}e^{3},\\
e^{3}(m^{+})^{3}_{3}=q^{\frac{2}{3}} (m^{+})^{3}_{3}e^{3},
\end{array}
\right\}
\Longrightarrow
\left\{
\begin{array}{l}
e^{3}K_{1}=K_{1}e^{3},\\
e^{3}K_{2}=q^{-1} K_{2}e^{3}.
\end{array}
\right.
\Longleftrightarrow
\left\{
\begin{array}{l}
E_{3}K_{1}=K_{1}E_{3},\\
E_{3}K_{2}=q^{-1} K_{2}E_{3}.
\end{array}
\right.
$$

In order to explore the relations between $F_{3}$ and $K_{i},i=1,2,3$,
we have
$K_{3}F_{3}=(m^{+})^{3}_{3}c^{-1}f_{3}=\frac{1}{\lambda}(m^{+})^{3}_{3}f_{3}c^{-1}
=\frac{1}{\lambda}\lambda f_{3}(m^{+})^{3}_{3}R^{33}_{33}c^{-1}=f_{3}(m^{+})^{3}_{3}c^{-1}R^{33}_{33}
=q^{2}F_{3}K_{3}
$
through the cross relation
$
(m^{+})^{i}_{i}f_{3}=\lambda f_{b}(m^{+})^{i}_{a}R^{ab}_{i3}.
$
Similarly,
$$
\left.
\begin{array}{l}
(m^{+})^{i+1}_{i+1}K_{i}=(m^{+})^{i}_{i},\ i=1, 2,\\
(m^{+})^{i}_{i}f_{3}=q^{-\frac{1}{3}}f_{3}(m^{+})^{i}_{i}, \ i=1, 2,\\
%m^{+2}_{2}f_{3}=q^{-\frac{1}{3}}f_{3}m^{+2}_{2}\\
(m^{+})^{3}_{3}f_{3}=q^{\frac{2}{3}}f_{3}(m^{+})^{3}_{3}.
\end{array}
\right\}
\Longrightarrow
\left\{
\begin{array}{l}
f_{3}K_{1}=K_{1}f_{3},\\
f_{3}K_{2}=q K_{2}f_{3}.
\end{array}
\right.
\Longleftrightarrow
\left\{
\begin{array}{l}
F_{3}K_{1}=K_{1}F_{3},\\
F_{3}K_{2}=q K_{2}F_{3}.
\end{array}
\right.
$$

Then the relations between $K_{3}$ and $E_{i},F_{i},i=1,2$ are given as follows
$$
E_{1}K_{3}
=E_{1}K^{-\frac{1}{3}}_{1}K^{-\frac{2}{3}}_{2}c^{-1}
=q^{-\frac{2}{3}}q^{\frac{2}{3}}K^{-\frac{1}{3}}_{1}K^{-\frac{2}{3}}_{2}E_{1}c^{-1}
=K^{-\frac{1}{3}}_{1}K^{-\frac{2}{3}}_{2}c^{-1}E_{1}
=K_{3}E_{1},
$$
$$
E_{2}K_{3}
=E_{2}K^{-\frac{1}{3}}_{1}K^{-\frac{2}{3}}_{2}c^{-1}
=q^{\frac{1}{3}}q^{-\frac{4}{3}}K^{-\frac{1}{3}}_{1}K^{-\frac{2}{3}}_{2}E_{2}c^{-1}
=q^{-1}K^{-\frac{1}{3}}_{1}K^{-\frac{2}{3}}_{2}c^{-1}E_{2}
=q^{-1}K_{3}E_{2},
$$
$$
F_{1}K_{3}
=F_{1}K^{-\frac{1}{3}}_{1}K^{-\frac{2}{3}}_{2}c^{-1}
=q^{\frac{2}{3}}q^{-\frac{2}{3}}K^{-\frac{1}{3}}_{1}K^{-\frac{2}{3}}_{2}F_{1}c^{-1}
=K^{-\frac{1}{3}}_{1}K^{-\frac{2}{3}}_{2}c^{-1}F_{1}
=K_{3}F_{1},
$$
$$
F_{2}K_{3}
=F_{2}K^{-\frac{1}{3}}_{1}K^{-\frac{2}{3}}_{2}c^{-1}
=q^{-\frac{1}{3}}q^{\frac{4}{3}}K^{-\frac{1}{3}}_{1}K^{-\frac{2}{3}}_{2}F_{2}c^{-1}
=qK^{-\frac{1}{3}}_{1}K^{-\frac{2}{3}}_{2}c^{-1}F_{2}
=qK_{3}F_{2}.
$$

Moreover,
$\Delta(E_{3})=E_{3}\otimes K_{3}+1\otimes E_{3}$,
$\Delta(F_{3})=F_{3}\otimes 1+K_{3}^{-1}\otimes F_{3}$
can be obtained by Corollary \ref{cor1}.
The most important relations are the $q$-Serre relations.
Since the generators $E_{i}$'s belong to $(m^{+})^{i}_{i+1}, i=1, 2$,
we will consider the following equalities
$$
\left.
\begin{array}{l}
(m^{+})^{1}_{2}=(q-q^{-1})E_{1}(m^{+})^{2}_{2},\\
e^{3}(m^{+})^{1}_{2}=\lambda R^{13}_{ab}(m^{+})^{a}_{2}e^{b}=\lambda q(m^{+})^{1}_{2}e^{3},\\
e^{3}(m^{+})^{2}_{2}=\lambda R^{23}_{ab}(m^{+})^{a}_{2}e^{b}=\lambda q(m^{+})^{2}_{2}e^{3}.
\end{array}
\right\}
\Longrightarrow
e^{3}E_{1}=E_{1}e^{3},
\Longrightarrow
E_{3}E_{1}=E_{1}E_{3}.
$$
$$
\left.
\begin{array}{l}
(m^{+})^{2}_{3}=(q-q^{-1})E_{2}(m^{+})^{3}_{3},\\
e^{3}(m^{+})^{2}_{3}=\lambda R^{23}_{ab}(m^{+})^{a}_{3}e^{b}=\lambda q(m^{+})^{2}_{3}e^{3}+\lambda (q^{2}-1)(m^{+})^{3}_{3}e^{2},\\
e^{3}(m^{+})^{3}_{3}=\lambda R^{33}_{ab}(m^{+})^{a}_{2}e^{b}=\lambda q^{2}(m^{+})^{3}_{3}e^{3}.
\end{array}
\right\}
\Longrightarrow
e^{2}=e^{3}E_{2}-q^{-1}E_{2}e^{3}.
$$
So, we need to explore the relation between $e^{2}$ and $e^{3}$. Note that
$e^{2}e^{3}=R'{}^{32}_{ab}e^{a}e^{b}=q^{-2}R^{32}_{ab}e^{a}e^{b}=q^{-2}R^{32}_{32}e^{3}e^{2}=q^{-1}e^{3}e^{2}$,
then combining with $e^{2}=e^{3}E_{2}-q^{-1}E_{2}e^{3}$,
we obtain
$$
(E_{3})^{2}E_{2}-(q+q^{-1})E_{3}E_{2}E_{3}+E_{2}(E_{3})^{2}=0.
$$
On the other hand,
we need to know another relation between $e^{2}$ and $E_{2}$,
which can be explored by the following cross relations
$$
\left.
\begin{array}{l}
(m^{+})^{2}_{3}=(q-q^{-1})E_{2}(m^{+})^{3}_{3},\\
e^{2}(m^{+})^{2}_{3}
%=\lambda R^{2}_{2}{}^{2}_{2}(m^{+})^{2}_{3}e^{2}
=\lambda q^{2}(m^{+})^{2}_{3}e^{2},\\
e^{2}(m^{+})^{3}_{3}
%=\lambda R^{3}_{3}{}^{2}_{2}(m^{+})^{2}_{3}e^{3}
=\lambda q(m^{+})^3_3e^3,\\
e^{2}=e^{3}E_{2}-q^{-1}E_{2}e^{3}.
\end{array}
\right\}
\Longrightarrow
(E_{2})^{2}E_{3}-(q+q^{-1})E_{2}E_{3}E_{2}+E_{3}(E_{2})^{2}=0.
$$

The $q$-Serre relation of $F_{i},1\leq i\leq 3$ can be obtained similarly.
Then the resulting quantum group is just the quantized enveloping algebra $U_{q}({\mathfrak {sl}}_{4})$, where
$e^{i},f_{i},i=1,2$ can be expressed by the $q$-commutators with generators $E_{i},F_{i},i=1,2,3.$
\end{proof}
\begin{remark}
In the above concrete example, we find that it is not necessary to know the entire $m^{\pm}$-matrix for determining the structure of resulting quantum group, except for the diagonal and minor
diagonal entries. At least, corresponding to the vector
representation of $U_q(\mathfrak g)$, all the simple root vectors
$E_{i}, F_{i}$ and group-like elements $K_{i}$ are included in the diagonal and minor diagonal
entries. Other entries in $m^{\pm}$ are filled by non-simple root
vectors generated by $q$-commutators with generators
$E_{i}, F_{i}$ and $K_{i}$. Moreover, the explicit expressions of the
diagonal and minor diagonal entries in $m^{\pm}$ can be easily
obtained by Lemma \ref{lem1}.
\end{remark}
For the $A$ series,
the data $R, R^{\prime},m^{\pm}$ as above are deduced from vector representation,
which inspires us to start with the vector representation when considering the general rank-inductive construction.
The $R_{VV}$-matrix of vector representation $T_{V}$  satisfies the quadratic equation
$(PR_{VV}-q^{\frac{n-1}{n}}I)(PR_{VV}+q^{-\frac{n+1}{n}}I)=0$.
So setting $R=q^{\frac{n+1}{n}}R_{VV},R^{\prime}=q^{-2}R,$
then we have $(PR+I)(PR^{\prime}-I)=0,$
and
$
R^{ij}_{kl}
=qq^{\delta_{ij}}\delta_{ik}\delta_{jl}+(q^{2}-1)\delta_{il}\delta_{jk}\theta(j-i), \
$
where
$$\theta(k)=
\left\{
\begin{array}{lcl}
1&~~&k>0,\\
0&~~&k\leq 0.
\end{array}
\right.
$$

On the other hand,
according to Lemma \ref{lem1},
we obtain the following
\begin{lemma}\label{alemma}
Corresponding to the vector representation,
the diagonal and minor diagonal entries in $FRT$-matrix $m^{\pm}$ of $U_{q}^{ext}({\mathfrak {sl}}_{n})$ are given by
\begin{gather*}
(m^{+})^{i}_{i+1}=(q-q^{-1})
E_{i}K^{-\frac{1}{n}}_{1}
K^{-\frac{2}{n}}_{2}\cdots K^{-\frac{i-1}{n}}_{i-1}
K^{-\frac{i}{n}}_{i}K^{\frac{n-(i+1)}{n}}_{i+1}\cdots K^{\frac{n-(n-1)}{n}}_{n-1}, \quad 1\leq i\leq n-1,
\\
(m^{+})^{i}_{i}=K^{-\frac{1}{n}}_{1}K^{-\frac{2}{n}}_{2}\cdots K^{-\frac{i-1}{n}}_{i-1}
K^{\frac{n-i}{n}}_{i}\cdots K^{\frac{n-(n-1)}{n}}_{n-1}, \quad 1\leq i\leq n.
\\
(m^{-})^{i+1}_{i}=(q-q^{-1})
K^{\frac{1}{n}}_{1}K^{\frac{2}{n}}_{2}\cdots K^{\frac{i-1}{n}}_{i-1}
K^{\frac{i}{n}}_{i}K^{-\frac{n-(i+1)}{n}}_{i+1}\cdots K^{-\frac{n-(n-1)}{n}}_{n-1}F_{i},\quad 1\leq i\leq n-1,
\\
(m^{-})^{i}_{i}=K^{\frac{1}{n}}_{1}K^{\frac{2}{n}}_{2}\cdots K^{\frac{i-1}{n}}_{i-1}
K^{-\frac{n-i}{n}}_{i}\cdots K^{-\frac{n-(n-1)}{n}}_{n-1},\quad 1\leq i\leq n.
\end{gather*}
\end{lemma}
Obviously,
we observe that
$(m^{+})^{i}_{i}K_{i}^{-1}=(m^{+})^{i+1}_{i+1}$ for $1\leq i\leq n-1$ by the above lemma.
With these,
we have the following
\begin{theorem}\label{propA}
For type $A$, let $\lambda=q^{-\frac{n+1}{n}}$, and
identify $e^{n},f_{n},(m^{+})^{n}_{n}c^{-1}$ with the additional simple root vectors $E_{n},F_{n}$ and group-like element $K_{n}$.
Then the resulting quantum group $U(V^{\vee}(R_{21}^{-1},R^{\prime}),\widetilde{U_{q}^{ext}({\mathfrak {sl}}_{n})},V(R,R^{\prime}))$
is exactly the $U_{q}({\mathfrak {sl}}_{n+1})$ with $K_{i}^{\pm\frac{1}{n}}$ adjoined.
\end{theorem}
\begin{proof}
$(m^{+})^{n}_{n}=K^{-\frac{1}{n}}_{1}K^{-\frac{2}{n}}_{2}\cdots K^{-\frac{n-1}{n}}_{n-1}$
can be obtained by Lemma \ref{alemma}.
$
[E_{n},F_{n}]=\frac{K_{n}-K_{n}^{-1}}{q-q^{-1}},
$
$\Delta(E_{n})=E_{n}\otimes K_{n}+1\otimes E_{n}$,
and
$\Delta(F_{n})=F_{n}\otimes 1+K_{n}^{-1}\otimes F_{n}$
can be deduced easily from Corollary \ref{cor1}.
On the other hand,
we have the following cross relations by Corollary \ref{cor1}
$$
\left.
\begin{array}{l}
(m^{-})^{i+1}_{i}=(q-q^{-1})(m^{-})^{i+1}_{i+1}F_{i},\\
(m^{-})^{i+1}_{i}e^{n}=\lambda R^{n,i+1}_{n,i+1}e^{n}(m^{-})^{i+1}_{i},\\
(m^{-})^{i+1}_{i+1}e^{n}=\lambda R^{n,i+1}_{n,i+1}e^{n}(m^{-})^{i+1}_{i+1},
\end{array}
\right\}
\Longrightarrow
F_{i}e^{n}=e^{n}F_{i},
\Longrightarrow
[E_{n},F_{i}]=0, \
1\leq i \leq n-1.
$$
$$
\left.
\begin{array}{l}
(m^{+})^{i}_{i+1}=(q-q^{-1})E_{i}(m^{+})^{i+1}_{i+1},\\
(m^{+})^{i}_{i+1}f_{n}=\lambda f_{n}(m^{+})^{i}_{i+1}R^{i+1,n}_{i+1,n},\\
(m^{+})^{i+1}_{i+1}f_{n}=\lambda f_{n}(m^{+})^{i+1}_{i+1}R^{i+1,n}_{i+1,n},
\end{array}
\right\}
\Longrightarrow
E_{i}f_{n}=f_{n}E_{i},
\Longrightarrow
[E_{i},F_{n}]=0, \
1\leq i \leq n-1.
$$

With these,
we get
$$[E_{n},F_{i}]=\delta_{ni}\frac{K_{n}-K_{n}^{-1}}{q-q^{-1}},\quad
[E_{i},F_{n}]=\delta_{in}\frac{K_{n}-K_{n}^{-1}}{q-q^{-1}}.
\eqno{(A1)}
$$

The relations between $E_{n},\, F_{n}$ and $K_{i}$, $1\leq i\leq n$ also can be obtained by the cross relations in Corollary \ref{cor1},
$$
\left.
\begin{array}{l}
(m^{+})^{i}_{i}K_{i}^{-1}=(m^{+})^{i+1}_{i+1}, \ 1\leq i\leq n-1,\\
e^{n}(m^{+})^{j}_{j}=\lambda R^{j}_{j}{}^{n}_{n}(m^{+})^{j}_{j}e^{n},\\
R^{jn}_{jn}=q, \ 1\leq j\leq n-1,\
R^{nn}_{nn}=q^{2}.
\end{array}
\right\}
\Longrightarrow
\left\{
\begin{array}{l}
E_{n}K_{j}=K_{j}E_{n}, \ 1\leq j\leq n-2,\\
E_{n}K_{n-1}=q^{-1}K_{n-1}E_{n}.
\end{array}
\right.
\eqno{(A2)}
$$
$$
\left.
\begin{array}{l}
(m^{+})^{i}_{i}K_{i}^{-1}=(m^{+})^{i+1}_{i+1},\ 1\leq i\leq n-1,\\
(m^{+})^{i}_{i}f_{n}=\lambda f_{n}(m^{+})^{i}_{i}R^{in}_{in},\\
R^{jn}_{jn}=q, \ 1\leq j\leq n-1, \
R^{nn}_{nn}=q^{2}.
\end{array}
\right\}
\Longrightarrow
\left\{
\begin{array}{l}
F_{n}K_{j}=K_{j}F_{n}, \ 1\leq j\leq n-2,\\
F_{n}K_{n-1}=qK_{n-1}F_{n}.
\end{array}
\right.
\eqno{(A3)}
$$

In the $U_{q}({\mathfrak {sl}}_{n})$,
$
E_{i}K_{i}=q^{2}K_{i}E_{i},E_{i}K_{i\pm 1}=q^{-1}K_{i\pm 1}E_{i},E_{i}K_{j}=K_{j}E_{i},j\neq i\pm 1;
F_{i}K_{i}=q^{-2}K_{i}F_{i},F_{i}K_{i\pm 1}=qK_{i\pm 1}F_{i},F_{i}K_{j}=K_{j}F_{i},j\neq i\pm 1.
$
Then we get the relations between $K_{n}$ and $E_{i}, F_{i}$,
where$ 1\leq i\leq n-2$.
$$
\left\{
\begin{array}{l}
E_{i}K_{n}
=q^{\frac{i-1}{n}}q^{-\frac{2i}{n}}q^{\frac{i+1}{n}}
K^{-\frac{1}{n}}_{1}\cdots
K^{-\frac{n-1}{n}}_{n-1}c^{-1}E_{i}
=K_{n}E_{i},\\
F_{i}K_{n}
=q^{-\frac{i-1}{n}}q^{\frac{2i}{n}}q^{-\frac{i+1}{n}}
K^{-\frac{1}{n}}_{1}\cdots
K^{-\frac{n-1}{n}}_{n-1}c^{-1}F_{i}
=K_{n}F_{i},\\
E_{n-1}K_{n}
=q^{\frac{n-2}{n}}q^{-\frac{2(n-1)}{n}}
K^{-\frac{1}{n}}_{1}\cdots
K^{-\frac{n-2}{n}}_{n-2}K^{-\frac{n-1}{n}}_{n-1}c^{-1}E_{n-1}
=q^{-1}K_{n}E_{n-1},\\
F_{n-1}K_{n}
=q^{-\frac{n-2}{n}}q^{\frac{2(n-1)}{n}}
K^{-\frac{1}{n}}_{1}\cdots
K^{-\frac{n-2}{n}}_{n-2}K^{-\frac{n-1}{n}}_{n-1}c^{-1}F_{n-1}
=qK_{n}F_{n-1}.
\end{array}
\right.
\eqno{(A4)}
$$

We want to explore the relations between $e^{n}$ and $E_{i},1\leq i\leq n-1,$
and observe that $E_{i}$ just belongs to the entry $(m^{+})^{i}_{i+1}$,
so
$$
\left.
\begin{array}{l}
(m^{+})^{i}_{i+1}=(q-q^{-1})E_{i}(m^{+})^{i+1}_{i+1},\\
e^{n}(m^{+})^{i}_{i+1}=\lambda R^{in}_{in}(m^{+})^{i}_{i+1}e^{n},1\leq i\leq n-2,\\
e^{n}(m^{+})^{n-1}_{n}
=\lambda R^{n-1,n}_{n-1,n}(m^{+})^{n-1}_{n}e^{n}+\lambda R^{n-1,n}_{n,n-1}(m^{+})^{n}_{n}e^{n-1},\\
e^{n}(m^{+})^{i+1}_{i+1}=\lambda R^{i+1,n}_{i+1,n}(m^{+})^{i+1}_{i+1}e^{n},\\
R^{jn}_{jn}=q,1\leq j\leq n-1,R^{nn}_{nn}=q^{2},R^{n-1,n}_{n,n-1}=q^{2}-1.
\end{array}
\right\}
\Rightarrow
\left\{
\begin{array}{l}
E_{n}E_{i}=E_{i}E_{n},1\leq i\leq n-2,\\
e^{n-1}=e^{n}E_{n-1}-q^{-1}E_{n-1}e^{n}.
\end{array}
\right.
$$
$$
\left.
\begin{array}{l}
(m^{-})^{i+1}_{i}=(q-q^{-1})(m^{-})^{i+1}_{i+1}F_{i},\\
f_{n}(m^{-})^{i+1}_{i}=\lambda R^{ni}_{ni}(m^{-})^{i+1}_{i}f_{n},1\leq i\leq n-2,\\
f_{n}(m^{-})^{n}_{n-1}=\lambda R^{n,n-1}_{n,n-1}(m^{-})^{n}_{n-1}f_{n}+\lambda R^{n-1,n}_{n,n-1}(m^{-})^{n}_{n}f_{n-1},\\
f_{n}(m^{-})^{i+1}_{i+1}=\lambda R^{n,i+1}_{n,i+1}(m^{-})^{i+1}_{i+1}f_{n},\\
\end{array}
\right\}
\Rightarrow
\left\{
\begin{array}{l}
F_{n}F_{i}=F_{i}F_{n},1\leq i\leq n-2,\\
f_{n-1}=qf_{n}F_{n-1}-F_{n-1}f_{n}.
\end{array}
\right.
$$
Note that
$
e^{n-1}e^{n}=R'{}^{n,n-1}_{ab}e^{a}e^{b}=q^{-2}R^{n,n-1}_{n,n-1}e^{n}e^{n-1}=q^{-1}e^{n}e^{n-1}
$
and
$
f_{n-1}f_{n}=f_{b}f_{a}R'{}^{ab}_{n-1,n}=q^{-2}f_{n}f_{n-1}R^{n-1,n-1}_{n-1,n}=q^{-1}f_{n}f_{n-1}.
$
Then we obtain
$$
\left\{
\begin{array}{l}
(E_{n})^{2}E_{n-1}-(q+q^{-1})E_{n}E_{n-1}E_{n}+E_{n-1}(E_{n})^{2}=0,\\
(F_{n})^{2}F_{n-1}-(q+q^{-1})F_{n}F_{n-1}F_{n}+F_{n-1}(F_{n})^{2}=0.
\end{array}
\right.
\eqno{(A5)}
$$

On the other hand,
$$
\left.
\begin{array}{l}
(m^{+})^{n-1}_{n}=(q-q^{-1})E_{n-1}(m^{+})^{n}_{n},\\
e^{n-1}(m^{+})^{n-1}_{n}=\lambda q^{2}(m^{+})^{n-1}_{n}e^{n-1},\\
e^{n-1}(m^{+})^{n}_{n}=\lambda q(m^{+})^{n}_{n}e^{n-1}.
\end{array}
\right\}
\Longrightarrow
e^{n-1}E_{n-1}=qE_{n-1}e^{n-1}.
$$
$$
\left.
\begin{array}{l}
(m^{-})^{n}_{n-1}=(q-q^{-1})(m^{-})^{n}_{n}F_{n-1},\\
f_{n-1}(m^{-})^{n}_{n-1}=\lambda q^{2}(m^{-})^{n}_{n-1}f_{n-1},\\
f_{n-1}(m^{-})^{n}_{n}=\lambda q(m^{-})^{n}_{n}f_{n-1}.
\end{array}
\right\}
\Longrightarrow
f_{n-1}F_{n-1}=qF_{n-1}f_{n-1}.
$$
Combining with $e^{n-1}=e^{n}E_{n-1}-q^{-1}E_{n-1}e^{n}$
and
$f_{n-1}=qf_{n}F_{n-1}-F_{n-1}f_{n}$,
we get
$$
\left\{
\begin{array}{l}
(E_{n-1})^{2}E_{n}-(q+q^{-1})E_{n-1}E_{n}E_{n-1}+E_{n}(E_{n-1})^{2}=0,\\
(F_{n-1})^{2}F_{n}-(q+q^{-1})F_{n-1}F_{n}F_{n-1}+F_{n}(F_{n-1})^{2}=0.
\end{array}
\right.
\eqno{(A6)}
$$
The other elements $e^{i},f_{j}$ can be identified with non-simple root vectors generated by $q$-commutators with $E_{i},F_{i},K_{i},1\leq i\leq n$.
With the above equalities $(A1)$---$(A6)$,
we prove that the resulting quantum groups is $U_{q}({\mathfrak {sl}}_{n+1})$.
\end{proof}

\subsection{Inductive construction of $U_q(\mathfrak g)$ for the $BCD$ series}
The rank-inductive construction of $U_{q}({\mathfrak {sl}}_{n})$ gives us the confidence to consider the $BCD$ series.
Their Dynkin diagrams are given respectively by the following diagrams,
and the arrow is point to the shorter of the two roots in the diagrams.

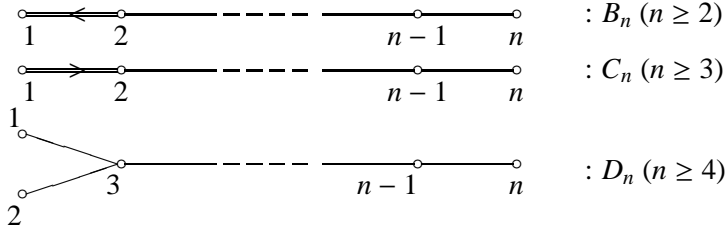
\setlength{\unitlength}{1mm}
\begin{picture}(98,6)
\put(6,4){\circle{1}}
\put(6.5,4.2){\line(1,0){12}}
\put(6.5,3.8){\line(1,0){12}}
\put(12,3){$<$}
\put(6,0){$1$}
\put(19,4){\circle{1}}
\put(19.5,4){\line(1,0){12}}
\put(18,0){$2$}
\multiput(32.5,4)(3,0){4}{\line(1,0){2}}
\put(45.5,4){\line(1,0){12}}
\put(58,4){\circle{1}}
\put(54,0){$n-1$}
\put(58.5,4){\line(1,0){12}}
\put(71,4){\circle{1}}
\put(70,0){$n$}
\put(80,3){$:B_{n} \ (n\geq2)$}
\end{picture}

\begin{picture}(98,6)
\put(6,4){\circle{1}}
\put(6.5,4.2){\line(1,0){12}}
\put(6.5,3.8){\line(1,0){12}}
\put(12,3){$>$}
\put(6,0){$1$}
\put(19,4){\circle{1}}
\put(19.5,4){\line(1,0){12}}
\put(18,0){$2$}
\multiput(32.5,4)(3,0){4}{\line(1,0){2}}
\put(45.5,4){\line(1,0){12}}
\put(58,4){\circle{1}}
\put(54,0){$n-1$}
\put(58.5,4){\line(1,0){12}}
\put(71,4){\circle{1}}
\put(70,0){$n$}
\put(80,3){$:C_{n} \ (n\geq3)$}
\end{picture}

\setlength{\unitlength}{1mm}
\begin{picture}(98,15)
\put(6,12){\circle{1}}
\put(18.5,8){\line(-3,1){12}}
\put(6,4){\circle{1}}
\put(6.5,4){\line(3,1){12}}
\put(4,13){$1$}
\put(4,0){$2$}
\put(19,8){\circle{1}}
\put(19.5,8){\line(1,0){12}}
\put(17,4){$3$}
\multiput(32.5,8)(3,0){4}{\line(1,0){2}}
\put(45.5,8){\line(1,0){12}}
\put(58,8){\circle{1}}
\put(58.5,8){\line(1,0){12}}
\put(50,4){$n-1$}
\put(71,8){\circle{1}}
\put(70,4){$n$}
\put(80,6){$:D_{n} \ (n\geq4)$}
\end{picture}

Corresponding to their vector representations,
the matrix $PR_{VV}$ satisfies the following cubic equation \cite{FRT1,klim}:
$$
(PR_{VV}+q^{-1}I)(PR_{VV}-qI)(PR_{VV}-\epsilon q^{\epsilon-N}I)=0.
$$
Then set
$
R=qR_{VV},
R^{\prime}=RPR-(\epsilon q^{\epsilon-N+1}+q^{2})R+(\epsilon q^{\epsilon-N+3}+1)P,
$
we get $(PR+I)(PR^{\prime}-I)=0$.
Moreover,
there is a unique formula for the matrix entries of $R$,
$$R^{ij}_{kl}=qq^{\delta_{ji}-\delta_{ji^{\prime}}}\delta_{ik}\delta_{jl}
+(q^{2}-1)\theta(j-l)(\delta_{il}\delta_{jk}-K^{ij}_{lk}).
$$
Here
$
K^{ij}_{lk}=\epsilon C^{i}_{j}C^{l}_{k},
$
and $C^{m}_{t}=\epsilon_{m}\delta_{mt^{\prime}}q^{-\rho_{m}},$
where
$i^{\prime}=N+1-i$,
let $N=2n$ if $N$ is even,
and $N=2n+1$ if $N$ is odd.
$\rho_{i}=\frac{N}{2}-i$ if $i<i^{\prime}$;
$\rho_{i^{\prime}}=-\rho_{i}$ if $i\leq i^{\prime}$,
$\epsilon=\epsilon_{1}=\cdots=\epsilon_{N}=1$ for $g={\mathfrak {so}}_{N},$
and
$
\rho_{i}=\frac{N}{2}+1-i,\rho_{i^{\prime}}=-\rho_{i}
$
if $i<i^{\prime}$,
$\epsilon_{1}=\cdots=\epsilon_{n}=1,
\epsilon=\epsilon_{n+1}=\cdots=\epsilon_{N}=-1$ for
$g={\mathfrak {sp}}_{N}$.

For their vector representations,
the diagonal and minor diagonal entries we need in $m^{\pm}$ can be obtained by Lemma \ref{lem1}.
\begin{lemma}\label{lemmabcd}
$(1)$
For $U_{q}^{ext}({\mathfrak {so}}_{2n+1}):$
\begin{gather*}
(m^{+})^{i}_{i}=K_{1}K_{2}\cdots K_{n-i}K_{n+1-i}, \quad (m^{+})^{n+1}_{n+1}=1, \quad 1\leq i\leq n,
\\
(m^{+})^{i}_{i+1}=-(q{-}q^{-1})E_{n+1-i}K_{1}K_{2}\cdots K_{n-i}, \quad 1\leq i\leq n{-}1, \quad
(m^{+})^{n}_{n+1}=-c_{0}E_{1},
\\
(m^{-})^{i+1}_{i}=(q{-}q^{-1})K_{1}^{-1}K_{2}^{-1}\cdots K_{n-i}^{-1}F_{n+1-i}, \quad
m^{-(n+1)}_{n}=c_{0}F_{1}, \quad (m^{+})^{i}_{i}(m^{-})^{i}_{i}=1,
\end{gather*}
where
$
c_{0}=(q^{\frac{1}{2}}+q^{-\frac{1}{2}})^{\frac{1}{2}}(q^{\frac{1}{2}}-q^{-\frac{1}{2}}).
$

$(2)$
For
$
U_{q}^{ext}({\mathfrak {sp}}_{2n}):
$
\begin{gather*}
(m^{+})^{i}_{i}=K_{1}^{\frac{1}{2}}K_{2}\cdots K_{n+1-i}, \quad
(m^{+})^{i^{\prime}}_{i^{\prime}}(m^{+})^{i}_{i}=1, \quad 1\leq i\leq n,
\\
(m^{+})^{i}_{i+1}=-(q{-}q^{-1})E_{n+1-i}K_{1}^{\frac{1}{2}}K_{2}\cdots K_{n-i}, \quad 1\leq i\leq n{-}1, \quad
(m^{+})^{n}_{n+1}=-(q^{2}{-}q^{-2})E_{1}K_{1}^{-\frac{1}{2}},
\\
(m^{-})^{i}_{i}=K_{1}^{-\frac{1}{2}}K_{2}^{-1}\cdots K_{n+1-i}^{-1}, \quad (m^{-})^{i^{\prime}}_{i^{\prime}}(m^{-})^{i}_{i}=1, \quad 1\leq i\leq n,
\\
(m^{-})^{i+1}_{i}=(q{-}q^{-1})K_{1}^{-\frac{1}{2}}K_{2}^{-1}\cdots K_{n-i}^{-1}F_{n+1-i}, \quad 1\leq i\leq n{-}1, \quad
(m^{-})^{n+1}_{n}=(q^{2}{-}q^{-2})K_{1}^{\frac{1}{2}}F_{1}.
\end{gather*}

$(3)$
For
$
U_{q}^{ext}({\mathfrak {so}}_{2n}):
$
\begin{gather*}
(m^{+})^{i}_{i}=(K_{1}^{\frac{1}{2}}K_{2}^{\frac{1}{2}})K_{3}\cdots K_{n+1-i}, \quad
(m^{+})^{i^{\prime}}_{i^{\prime}}(m^{+})^{i}_{i}=1, \quad 1\leq i\leq n{-}2,
\\
(m^{+})^{n-1}_{n-1}=K_{1}^{\frac{1}{2}}K_{2}^{\frac{1}{2}}, \quad
(m^{+})^{n}_{n}=K_{1}^{\frac{1}{2}}K_{2}^{-\frac{1}{2}}, \quad
(m^{+})^{n-1}_{n+1}=-(q{-}q^{-1})E_{1}(K_{1}^{-\frac{1}{2}}K_{2}^{\frac{1}{2}}),
\\
(m^{+})^{i}_{i+1}=-(q{-}q^{-1})E_{n+1-i}(K_{1}^{\frac{1}{2}}K_{2}^{\frac{1}{2}})K_{3}\cdots K_{n-i}, \quad 1\leq i\leq n{-}1,
\\
(m^{-})^{i}_{i}=(K_{1}^{-\frac{1}{2}}K_{2}^{-\frac{1}{2}})K_{3}^{-1}\cdots K_{n+1-i}^{-1}, \quad (m^{-})^{i^{\prime}}_{i^{\prime}}(m^{-})^{i}_{i}=1, \quad 1\leq i\leq n{-}2,
\\
(m^{-})^{n-1}_{n-1}=K_{1}^{-\frac{1}{2}}K_{2}^{-\frac{1}{2}}, \quad
(m^{-})^{n}_{n}=K_{1}^{-\frac{1}{2}}K_{2}^{\frac{1}{2}}, \quad
(m^{-})^{n+1}_{n-1}=(q{-}q^{-1})(K_{1}^{\frac{1}{2}}K_{2}^{-\frac{1}{2}})F_{1},
\\
(m^{-})^{i+1}_{i}=(q{-}q^{-1})(K_{1}^{-\frac{1}{2}}K_{2}^{-\frac{1}{2}})K_{3}^{-1}\cdots K_{n-i}^{-1}F_{n+1-i}, \quad 1\leq i\leq n{-}1.
\end{gather*}
\end{lemma}
With these,
we have the following
\begin{theorem}
With quantum group normalization constant $\lambda=q^{-1}$.

$(1)$ Type $B$

\noindent
Identify $e^{2n+1},f_{2n+1},(m^{+})^{2n+1}_{2n+1}c^{-1}$ with the additional simple root vectors $E_{n+1}$, $F_{n+1}$ and the group-like element $K_{n+1}$.
Then the resulting quantum group
$U(V^{\vee}(R^{\prime},R_{21}^{-1}),\widetilde{U_{q}^{ext}({\mathfrak {so}}_{2n+1})}$, $V(R^{\prime},R))$
is the quantum group
$U_{q}({\mathfrak {so}}_{2n+3})$.

$(2)$ Type $C$

\noindent
Identify $e^{2n},f_{2n},(m^{+})^{2n}_{2n}c^{-1}$ with the additional simple root vectors $E_{n+1},F_{n+1}$ and the group-like element $K_{n+1}$.
Then the resulting quantum group
$U(V^{\vee}(R^{\prime},R_{21}^{-1}),\widetilde{U_{q}^{ext}({\mathfrak {sp}}_{2n})},V(R^{\prime},R))$
is the quantum group
$U_{q}({\mathfrak {sp}}_{2n+2})$
with $K_1^{\pm\frac{1}{2}}$ adjoined.

$(3)$ Type $D$

\noindent
Identify $e^{2n},f_{2n},(m^{+})^{2n}_{2n}c^{-1}$ with the additional simple root vectors $E_{n+1}, F_{n+1}$ and the group-like element $K_{n+1}$.
Then the resulting quantum group
$U(V^{\vee}(R^{\prime},R_{21}^{-1}),\widetilde{U_{q}^{ext}({\mathfrak {so}}_{2n})},V(R^{\prime},R))$
is the quantum group
$U_{q}({\mathfrak {so}}_{2n+2})$
with
$K_1^{\pm\frac{1}{2}}K_2^{\pm\frac{1}{2}}$,
$K_1^{\pm\frac{1}{2}}K_2^{\mp\frac{1}{2}}$
adjoined.
%with set $E_{n}=e^{2n},F_{n}=f_{2n},K_{n}=m^{+2n}_{2n}c^{-1}=K_{1}^{-1}\cdots K_{n-2}^{-1}K_{n-1}^{-\frac{1}{2}}K_{n}^{-\frac{1}{2}}c^{-1}.$
\end{theorem}
\begin{proof}
The proof of Theorem \ref{propA} means that the relations of negative part can be obtained in a similar way,
so we only focus on the relations of the positive part.

$(1)$
$(m^{+})^{2n+1}_{2n+1}c^{-1}=K_{1}^{-1}\cdots K_{n-1}^{-1}K_{n}^{-1}$ follows from Lemma \ref{lemmabcd}.
From the identification in the above theorem, it is easily deduced from Corollary \ref{cor1} that
$[E_{n+1},F_{n+1}]=\frac{K_{n+1}-K_{n+1}^{-1}}{q-q^{-1}}$,
$\Delta(E_{n+1})=E_{n+1}\otimes K_{n+1}+1\otimes E_{n+1}$,
and
$\Delta(F_{n+1})=F_{n+1}\otimes 1+K_{n+1}^{-1}\otimes F_{n+1}$.

On the other hand, we have
$E_{n+1}K_{n+1}=e^{2n+1}(m^{+})^{2n+1}_{2n+1}c^{-1}
=\lambda R^{2n+1}_{a}{}^{2n+1}_{b}(m^{+})^{a}_{2n+1}e^{b}c^{-1}
$ $=R^{2n+1}_{2n+1}{}^{2n+1}_{2n+1}(m^{+})^{2n+1}_{2n+1}c^{-1}e^{2n+1}
=q^{2}K_{n+1}E_{n+1}.
$
The relations between the additional simple root vector $e^{2n+1}$ and other $K_{i},1\leq i\leq n$ can be deduced from
$e^{2n+1}(m^{+})^{p}_{p}=\lambda R^{p,2n+1}_{ab}(m^{+})^{a}_{p}e^{b}$ $=\lambda R^{p,2n+1}_{p,2n+1}(m^{+})^{p}_{p}e^{2n+1}, \ 1\leq p\leq n+1$.
Combining with
$
(m^{+})^{i+1}_{i+1}K_{n+1-i}=(m^{+})^{i}_{i}, \
1\leq i\leq n,
$
and
$
R^{1}_{1}{}^{2n+1}_{2n+1}=1, \
R^{p}_{p}{}^{2n+1}_{2n+1}=q, \
2\leq p\leq n+1,
$
we obtain
$
e^{2n+1}K_{n}=q^{-1}K_{n}e^{2n+1}, \
e^{2n+1}K_{i}=K_{i}e^{2n+1}$, for
$1\leq i\leq n-1,
$
namely,
$$
E_{n+1}K_{n}=q^{-1}K_{n}E_{n+1}, \quad
E_{n+1}K_{i}=K_{i}E_{n+1}, \quad
1\leq i\leq n-1.
$$
We observe that $F_{i}$ belongs to $(m^{-})^{i+1}_{i}$,
so the relations between $E_{n+1}$ and $F_{i}$ can be obtained by the equality
$(m^{-})^{i}_{j}e^{k}=\lambda R^{k}_{a}~{}^{i}_{b}e^{a}(m^{-})^{b}_{j}$
in Corollary \ref{cor1}.
$$
\left.
\begin{array}{l}
(m^{-})^{i+1}_{i}=(q-q^{-1})(m^{-})^{i+1}_{i+1}F_{i}, \ 1\leq i\leq n-1,\\
(m^{-})^{i+1}_{i}e^{2n+1}=\lambda R^{2n+1}_{a}{}^{i+1}_{b}e^{a}(m^{-})^{b}_{i}=e^{2n+1}(m^{-})^{i+1}_{i},1\leq i\leq n,\\
(m^{-})^{p}_{p}e^{2n+1}=\lambda R^{2n+1}_{a}{}^{p}_{b}e^{a}(m^{-})^{b}_{p}=e^{2n+1}(m^{-})^{p}_{p}, \ 2\leq p\leq n,
\end{array}
\right\}
\Longrightarrow
[E_{n+1},F_{i}]=0, \ 1\leq i\leq n.
$$

We will explore the $q$-Serre relations of the positive part.
We also observe that
$E_{n+1-i}$ belongs to $(m^{+})^{i}_{i+1}, \ 1\leq i\leq n,$
so
$$
\left\{
\begin{array}{l}
e^{2n+1}(m^{+})^{1}_{2}=\lambda R^{1}_{a}{}{}^{2n+1}_{b}(m^{+})^{a}_{2}e^{b}
=\lambda R^{1}_{1}{}{}^{2n+1}_{2n+1}(m^{+})^{1}_{2}e^{2n+1}+\lambda R^{1}_{2}{}{}^{2n+1}_{2n}(m^{+})^{2}_{2}e^{2n},\\
e^{2n+1}(m^{+})^{i}_{i+1}=\lambda R^{i}_{a}{}{}^{2n+1}_{b}(m^{+})^{a}_{i+1}e^{b}
=\lambda R^{i}_{i}{}{}^{2n+1}_{2n+1}(m^{+})^{i}_{i+1}e^{2n+1}
=\lambda q(m^{+})^{i}_{i+1}e^{2n+1}, \ 2\leq i\leq n.
\end{array}
\right.
$$
Putting the expression of $(m^{+})^{i}_{i+1}, \ 1\leq i\leq n$ into the above equalities,
we get
$$
\left\{
\begin{array}{l}
e^{2n+1}E_{j}=E_{j}e^{2n+1}, \ 1\leq j\leq n-1,\\
e^{2n}=e^{2n+1}E_{n}-q^{-1}E_{n}e^{2n+1}.
\end{array}
\right.
$$

So, we need to know the relations between $e^{2n}$ and $e^{2n+1},E_{n}$.
We have
$e^{2n+1}e^{2n}=R^{\prime}{}^{2n}_{a}{}^{2n+1}_{b}e^{a}e^{b}=-(q+\epsilon q^{\epsilon-N+2})e^{2n}e^{2n+1}+(\epsilon q^{\epsilon-N+1}+2)e^{2n+1}e^{2n}$,
so $e^{2n+1}e^{2n}=qe^{2n}e^{2n+1}$.
Combining with
$
e^{2n}=e^{2n+1}E_{n}-q^{-1}E_{n}e^{2n+1},
$
we get
$
(e^{2n+1})^{2}E_{n}-(q+q^{-1})e^{2n+1}E_{n}e^{2n+1}+E_{n}(e^{2n+1})^{2}=0,
$
namely,
$$
(E_{n+1})^{2}E_{n}-(q+q^{-1})E_{n+1}E_{n}E_{n+1}+E_{n}(E_{n+1})^{2}=0.
$$
On the other hand,
according to the equality
$e^{2n}(m^{+})^{1}_{2}=\lambda R^{1}_{a}{}^{2n}_{b}(m^{+})^{a}_{2}e^{b}=\lambda R^{1}_{1}{}^{2n}_{2n}(m^{+})^{1}_{2}e^{2n}=(m^{+})^{1}_{2}e^{2n}$
and
$
e^{2n}(m^{+})^{2}_{2}=\lambda R^{2}_{2}{}^{2n}_{2n}(m^{+})^{2}_{2}e^{2n}=q^{-1}(m^{+})^{2}_{2}e^{2n},
$
we get
$e^{2n}E_{n}=qE_{n}e^{2n}$.
Combining with
$e^{2n}=e^{2n+1}E_{n}-q^{-1}E_{n}e^{2n+1}$
again,
we obtain
$$
(E_{n})^{2}E_{n+1}-(q+q^{-1})E_{n}E_{n+1}E_{n}+E_{n+1}(E_{n})^{2}=0.
$$
With these relations,
we prove that the resulting quantum group is $U_{q}({\mathfrak {so}}_{2n+3})$.
$(2)$ and $(3)$ can be proved in a similar way.
\end{proof}

\section{Type-crossing constructions of types $BCD$ starting from type $A$}

Up to now, we have got the general inductive constructions of the classical quantum groups within the same type as in the above section.
However, Majid claimed that his double-bosonization allows to create not only a line of nodes diagram but also a tree of nodes diagram of quantum groups.
At each node of the tree,
we have more choices to adjoin suitable braided groups to obtain possible different new quantum groups of higher rank one.
In this section, we will give some examples to demonstrate this fact.
As we known, at the source node corresponding to $U_{q}({\mathfrak {sl}}_{2})$, Majid \cite{majid3} chose a pair of braided groups generated by the vector representation of $U_{q}({\mathfrak {sl}}_{2})$ to give $U_{q}({\mathfrak {sl}}_{3})$ as above.
In what follows, we will give $3$ kinds of type-crossing constructions: from type $A_1$ to type $B_2$, from type $A_2$ to type $C_3$, and from type $A_3$ to type $D_4$.

Corresponding to an irreducible representation of $U_{q}(\mathfrak g)$,
the matrix $R_{VV}$ is induced by
$R_{VV}=B_{VV}\circ(T_{V}\otimes T_{V})(\mathfrak{R})$.
Here
$
\mathfrak{R}=\sum\limits_{r_{1},\cdots,r_{n}=0}^{\infty}\prod\limits_{j=1}^{n}
\frac{(1-q_{\beta_{j}}^{-2})^{r_{j}}}{[r_{j}]_{q_{\beta_{j}}}!}q_{\beta_{j}}^{\frac{r_{j}(r_{j}+1)}{2}}E_{\beta_{j}}^{r_{j}}\otimes F_{\beta_{j}}^{r_{j}}
$
is the main part of the universal $R$-matrix of $U_h(\mathfrak g)$.
$B_{VV}$ denotes the linear operator on $V\otimes V$ given by
$B_{VV}(v\otimes w):=q^{(\mu,\mu^{\prime})}v\otimes w$
for
$v\in V_{\mu},$
$w\in V_{\mu^{\prime}}$.

\subsection{Type-crossing construction from type $A_1$ to type $B_2$} Now,
we still start from the node diagram of $U_{q}({\mathfrak {sl}}_{2})$ and
choose other braided groups to give $U_{q}({\mathfrak {so}}_{5})$ in the following.
\begin{example}
Starting from a $3$-dimensional representation $T_{V}$ of $U_{q}({\mathfrak {sl}}_{2})$,
which is given by
$
E_{1}(x_{i})=[2]_{q}x_{i+1},E_{1}(x_{3})=0,
F_{1}(x_{i+1})=x_{i},F_{1}(x_{1})=0,i=1,2,
$
where $x_{1},x_{2},x_{3}$ is the base of $V$ with corresponding weights $-\alpha_{1},0,\alpha_{1}$.
Then we get a $9\times9$ $R_{VV}$-matrix datum
$$
R_{VV}
=
\left(
\begin{array}{ccccccccc}
q^{2}&0&0&0&0&0&0&0&0\\
0&1&0&q^{2}{-}q^{-2}&0&0&0&0&0\\
0&0&q^{-2}&0&q^{2}{-}q^{-2}&0&(1{-}q^{-2})(q^{2}{-}q^{-2})&0&0\\
0&0&0&1&0&0&0&0&0\\
0&0&0&0&1&0&1{-}q^{-4}&0&0\\
0&0&0&0&0&1&0&q^{2}{-}q^{-2}&0\\
0&0&0&0&0&0&q^{-2}&0&0\\
0&0&0&0&0&0&0&1&0\\
0&0&0&0&0&0&0&0&q^{2}
\end{array}
\right).
$$

Clearly, $R_{VV}$ is invertible,
and according to the submodules decomposition of module $V^{\otimes 2}$, it is easy to see that $PR_{VV}$ obeys the minimal polynomial
$$
(PR_{VV}+q^{-2}I)(PR_{VV}-q^{2}I)(PR_{VV}-q^{-4}I)=0.
$$
Setting
$R=q^{2}R_{VV},
R^{\prime}=RPR-q^{-2}R-q^{4}R+(q^{2}+1)P$,
then
$$(PR+I)(PR^{\prime}-I)=0.$$
With these $R$ and $R'$, choose a pair of braided groups $V^{\vee}(R',R_{21}^{-1}),V(R',R)$,
and $\lambda=q^{-2}$.
Identify $e^{3},f_{3},(m^{+})^{3}_{3}c^{-1}$ with the additional simple root vectors $E_{2}, F_{2}$ and the group-like element $K_{2}$.
Then the resulting quantum group $U(V^{\vee}(R',R_{21}^{-1})$, $\widetilde{U_{q}^{ext}({\mathfrak {sl}}_{2})},V(R',R))$
is the quantum group $U_{q}({\mathfrak {so}}_{5})$.
\end{example}
\begin{proof}
Corresponding to this representation,
we get the $m^{\pm}$-matrix as follows
$$
m^{+}=
\left(
\begin{array}{ccc}
K_{1}&-(q{-}q^{-1})E_{1}&\frac{q(1{-}q^{-2})^{2}}{[2]_{q}}E_{1}^{2}K_{1}^{-1}\\
0&1&-(q{-}q^{-1})E_{1}K_{1}^{-1}\\
0&0&K_{1}^{-1}
\end{array}
\right),
$$
$$
m^{-}=
\left(
\begin{array}{ccc}
K_{1}^{-1}&0&0\\
(q^{4}{-}1)F_{1}&1&0\\
(q^{4}{-}q^{2})(q^{4}{-}1)K_{1}F_{1}^{2}&(q^{2}{-}q^{-2})K_{1}F_{1}&K_{1}
\end{array}
\right).
$$

With the identification as above,
$
[E_{2},F_{2}]=\frac{K_{2}-K_{2}^{-1}}{q^{2}-q^{-2}},
$
$\Delta(E_{2})=E_{2}\otimes K_{2}+1\otimes E_{2}$,
and
$\Delta(F_{2})=F_{2}\otimes 1+K_{2}^{-1}\otimes F_{2}$
can be deduced from Corollary \ref{cor1}.
%can be received from
%$[e^{3},f_{3}]=\frac{(m^{+})^{3}_{3}c^{-1}-c(m^{-})^{3}_{3}}{q^{2}-q^{-2}}$.
%$

$
E_{2}K_{2}=e^{3}(m^{+})^{3}_{3}c^{-1}
=\lambda R^{3}_{3}{}^{3}_{3}(m^{+})^{3}_{3}e^{3}c^{-1}
=R^{3}_{3}{}^{3}_{3}(m^{+})^{3}_{3}c^{-1}e^{3}
=q^{4}K_{2}E_{2},
$
$
E_{2}K_{1}=q^{-2}K_{1}E_{2}$ can be deduced from
$
e^{3}(m^{+})^{1}_{1}=\lambda R^{1}_{1}{}^{3}_{3}(m^{+})^{1}_{1}e^{3}=q^{-2}(m^{+})^{1}_{1}e^{3}.
$
On the other hand,
$E_{1}K_{2}=E_{1}(m^{+})^{3}_{3}c^{-1}=E_{1}K_{1}^{-1}c^{-1}=q^{-2}K_{1}^{-1}c^{-1}E_{1}=q^{-2}K_{2}E_{1}$.
According to the equality $(m^{-})^{2}_{1}e^{3}=\lambda R^{3}_{3}{}^{2}_{2}e^{3}(m^{-})^{2}_{1}=e^{3}(m^{-})^{2}_{1}$,
we obtain $e^{3}F_{1}=F_{1}e^{3}$ combining with $(m^{-})^{2}_{1}=(q^{4}-1)F_{1}$,
namely,
$
[E_{2},F_{1}]=0.
$

We will explore the $q$-Serre relation between $e^{3}$ and $E_{1}$.
The equality
$(q+q^{-1})e^{2}=q^{-2}E_{1}e^{3}-e^{3}E_{1}$ can be given by
$
e^{3}(m^{+})^{1}_{2}=\lambda R^{13}_{ab}(m^{+})^{a}_{2}e^{b}=q^{-2}(m^{+})^{1}_{2}e^{3}+(q^{2}-q^{-2})(m^{+})^{2}_{2}e^{2}.
$
Combining with
$e^{3}e^{2}=q^{2}e^{2}e^{3}$,
which is deduced from $e^{2}e^{3}=R'{}^{32}_{ab}e^{a}e^{b}=(q^{4}+q^{2}+1)e^{2}e^{3}-(q^{2}+1)e^{3}e^{2}$,
we obtain
$$
(E_{2})^{2}E_{1}-(q^{2}+q^{-2})E_{2}E_{1}E_{2}+E_{1}(E_{2})^{2}=0.
$$
On the other hand,
we need to know the relation between $e^{2}$ and $E_{1}$.
$$
\left.
\begin{array}{l}
e^{2}(m^{+})^{1}_{2}=\lambda R^{12}_{ab}(m^{+})^{a}_{2}e^{b}=(m^{+})^{1}_{2}e^{2}+(q^{2}{-}q^{-2})(m^{+})^{2}_{2}e^{1},\\
e^{1}(m^{+})^{1}_{2}=\lambda R^{11}_{ab}(m^{+})^{a}_{2}e^{b}=\lambda R^{11}_{11}(m^{+})^{1}_{2}e^{1}=q^{2}(m^{+})^{1}_{2}e^{1}.
\end{array}
\right\}
\Longrightarrow
\left\{
\begin{array}{l}
e^{1}=\frac{1}{q+q^{-1}}(E_{1}e^{2}-e^{2}E_{1}),\\
e^{1}E_{1}=q^{2}E_{1}e^{1}.
\end{array}
\right.
$$
Combining with
$(q+q^{-1})e^{2}=q^{-2}E_{1}e^{3}-e^{3}E_{1}$,
we obtain
$$
(E_{1})^{3}E_{2}-
\left[
\begin{array}{c}
3\\
1
\end{array}
\right]
_{q}(E_{1})^{2}E_{2}E_{1}+
\left[
\begin{array}{c}
3\\
2
\end{array}
\right]
_{q}E_{1}E_{2}(E_{1})^{2}
-E_{2}(E_{1})^{3}=0.
$$
With these relations, the Cartan matrix of the resulting quantum group is
$
\left(
\begin{array}{cc}
2&-2\\
-1&2
\end{array}
\right)
$.

So the resulting quantum group is $U_{q}({\mathfrak {so}}_{5})$,
and the proof is complete.
\end{proof}

\subsection{Type-crossing construction from type $A_2$ to type $C_3$}
In Example 3.1, we get $U_{q}({\mathfrak {sl}}_{4})$ starting from $U_{q}({\mathfrak {sl}}_{3})$ by choosing the braided groups generated by a $3$-dimensional vector representation via the Majid's double-bosonization construction.
In the example below, we will choose another pair of braided groups generated by a $6$-dimensional irreducible module to give $U_{q}({\mathfrak {sp}}_{6})$ starting from the node diagram of $U_{q}({\mathfrak {sl}}_{3})$.
\begin{example}
The pair of braided groups we want to have is obtained from the $6$-dimensional irreducible representation $T_{V}$ of $U_{q}({\mathfrak {sl}}_{3})$,
which is defined by

$
E_{1}
\left(
\begin{array}{c}
x_{1}\\
x_{2}\\
x_{3}
\end{array}
\right)
=
\left(
\begin{array}{c}
x_{2}\\
(q+q^{-1})x_{4}\\
x_{5}
\end{array}
\right)
$
,\qquad
$
E_{2}
\left(
\begin{array}{c}
x_{2}\\
x_{4}\\
x_{5}
\end{array}
\right)
=
\left(
\begin{array}{c}
x_{3}\\
x_{5}\\
(q+q^{-1})x_{6}
\end{array}
\right),
$

$
F_{1}
\left(
\begin{array}{c}
x_{2}\\
x_{4}\\
x_{5}
\end{array}
\right)
=
\left(
\begin{array}{c}
(q+q^{-1})x_{1}\\
x_{2}\\
x_{3}
\end{array}
\right)
$
,\qquad
$
F_{2}
\left(
\begin{array}{c}
x_{3}\\
x_{5}\\
x_{6}
\end{array}
\right)
=
\left(
\begin{array}{c}
x_{2}\\
(q+q^{-1})x_{4}\\
x_{5}
\end{array}
\right),
$

\noindent
where $\{\,x_{i}\mid 1\leq i\leq 6\,\}$ is a basis of $V$ with the corresponding weights
$-2\lambda_{1},
-2\lambda_{1}+\alpha_{1},
-2\lambda_{1}+\alpha_{1}+\alpha_{2},
-2\lambda_{1}+2\alpha_{1},
-2\lambda_{1}+2\alpha_{1}+\alpha_{2},
-2\lambda_{1}+2\alpha_{1}+2\alpha_{2}$,
respectively.
We can obtain a $36\times 36$ matrix $R_{VV}$ corresponding to the $6$-dimensional representation,
and the $PR_{VV}$ obeys the minimal polynomial
$$
(PR_{VV}-q^{\frac{8}{3}}I)(PR_{VV}+q^{-\frac{4}{3}}I)(PR_{VV}-q^{-\frac{10}{3}}I)=0.
$$
Setting $R=q^{\frac{4}{3}}R_{VV},R^{\prime}=RPR-(q^{-2}+q^{4})R+(q^{2}+1)P$,
we have
$$(PR+I)(PR^{\prime}-I)=0,$$
and $\lambda=q^{-\frac{4}{3}}$.
Identify $e^{6},f_{6},(m^{+})^{6}_{6}c^{-1}$ with the additional simple root vectors $E_{3}, F_{3}$ and the group-like element $K_{3}$,
then the resulting quantum group $U(V^{\vee}(R^{\prime},R_{21}^{-1}),\widetilde{U_{q}^{ext}({\mathfrak {sl}}_{3})},V(R^{\prime},R))$
is the quantum group $U_{q}({\mathfrak {sp}}_{6})$ with $K_{i}^{\frac{1}{3}} \ (i=1, 2)$ adjoined.
\end{example}
\begin{proof}
The elements in matrix $m^{\pm}$ we need can be obtained by Lemma \ref{lem1},
which are given by the following equalities
$$
\left\{
\begin{array}{l}
(m^{+})^{1}_{2}=-(q-q^{-1})E_{1}K_{1}^{\frac{1}{3}}K_{2}^{\frac{2}{3}}, \quad
(m^{+})^{2}_{2}=K_{1}^{\frac{1}{3}}K_{2}^{\frac{2}{3}},\\
(m^{+})^{5}_{6}=-(q^{2}-q^{-2})E_{2}K_{1}^{-\frac{2}{3}}K_{2}^{-\frac{4}{3}}, \quad
(m^{+})^{6}_{6}=K_{1}^{-\frac{2}{3}}K_{2}^{-\frac{4}{3}}, \quad
(m^{+})^{4}_{4}=K_{1}^{-\frac{2}{3}}K_{2}^{\frac{2}{3}},\\
(m^{-})^{5}_{3}=q(q-q^{-1})K_{1}^{\frac{2}{3}}K_{2}^{\frac{1}{3}}F_{1}, \quad
(m^{-})^{5}_{5}=K_{1}^{\frac{2}{3}}K_{2}^{\frac{1}{3}},\\
(m^{-})^{6}_{5}=(q-q^{-1})K_{1}^{\frac{2}{3}}K_{2}^{\frac{4}{3}}F_{2}, \quad
(m^{-})^{6}_{6}=K_{1}^{\frac{2}{3}}K_{2}^{\frac{4}{3}}.
\end{array}
\right.
$$

We only focus on the relations of the positive part.

Note that
$
E_{3}K_{3}=e^{6}(m^{+})^{6}_{6}c^{-1}
=\lambda R^{66}_{66}(m^{+})^{6}_{6}e^{6}c^{-1}$
$=R^{66}_{66}(m^{+})^{6}_{6}c^{-1}e^{6}
=q^{4}(m^{+})^{6}_{6}c^{-1}e^{6}=q^{4}K_{3}E_{3}.
$
Combining with $e^{6}(m^{+})^{i}_{i}=\lambda R^{i}_{i}{}^{6}_{6}(m^{+})^{i}_{i}e^{6}$,
we have
$$
\left.
\begin{array}{l}
(m^{+})^{6}_{6}K_{2}^{2}=(m^{+})^{4}_{4},\\
e^{6}(m^{+})^{6}_{6}=q^{\frac{8}{3}}(m^{+})^{6}_{6}e^{6},\\
e^{6}(m^{+})^{4}_{4}=q^{-\frac{4}{3}}(m^{+})^{4}_{4}e^{6}.
\end{array}
\right\}
\Longrightarrow
\left\{
\begin{array}{l}
e^{6}K_{1}=K_{1}e^{6},\\
e^{6}K_{2}=q^{-2}K_{2}e^{6}.
\end{array}
\right.
\Longrightarrow
\left\{
\begin{array}{l}
E_{3}K_{1}=K_{1}E_{3},\\
E_{3}K_{2}=q^{-2}K_{2}E_{3}.
\end{array}
\right.
$$

On the other hand,
$
E_{2}K_{3}
=E_{2}K_{1}^{-\frac{2}{3}}K_{2}^{-\frac{4}{3}}c^{-1}
=q^{\frac{2}{3}}q^{-\frac{8}{3}}K_{1}^{-\frac{2}{3}}K_{2}^{-\frac{4}{3}}c^{-1}E_{2}
=q^{-2}K_{3}E_{2},
$
$
E_{1}K_{3}
=E_{1}K_{1}^{-\frac{2}{3}}K_{2}^{-\frac{4}{3}}c^{-1}
=q^{-\frac{4}{3}}q^{\frac{4}{3}}K_{1}^{-\frac{2}{3}}K_{2}^{-\frac{4}{3}}c^{-1}E_{1}
=K_{3}E_{1}.
$
We will explore the $q$-Serre relations. Since
$E_{1},E_{2}$ belong to $(m^{+})^{1}_{2}, \ (m^{+})^{5}_{6}$, respectively,
then
$$
\left.
\begin{array}{l}
(m^{+})^{1}_{2}=-(q-q^{-1})E_{1}(m^{+})^{2}_{2},\\
e^{6}(m^{+})^{1}_{2}=q^{-\frac{4}{3}}(m^{+})^{1}_{2}e^{6},\\
e^{6}(m^{+})^{2}_{2}=q^{-\frac{4}{3}}(m^{+})^{2}_{2}e^{6}.
\end{array}
\right\}
\Longrightarrow
e^{6}E_{1}=E_{1}e^{6},
\Longrightarrow
E_{3}E_{1}=E_{1}E_{3}.
$$
$$
\left.
\begin{array}{l}
(m^{+})^{5}_{6}=-(q^{2}-q^{-2})E_{2}(m^{+})^{6}_{6},\\
e^{6}(m^{+})^{5}_{6}=q^{\frac{2}{3}}(m^{+})^{5}_{6}e^{6}+(q^{2}-q^{-2})q^{\frac{2}{3}}(m^{+})^{6}_{6}e^{5},\\
e^{6}(m^{+})^{6}_{6}=q^{\frac{8}{3}}(m^{+})^{6}_{6}e^{6}.
\end{array}
\right\}
\Longrightarrow
e^{5}=q^{-\frac{4}{3}}(q^{-2}E_{2}e^{6}-e^{6}E_{2}).
$$

So we need to know the relation between $e^{5}$ and $e^{6}$.
$e^{6}e^{5}=q^{2}e^{5}e^{6}$ is obtained by
$e^{i}e^{j}=R'^{ji}_{ab}e^{a}e^{b}$ and $R^{\prime}{}^{5}_{5}{}^{6}_{6}=-q^{2}-1,R^{\prime}{}^{5}_{6}{}^{6}_{5}=2+q^{-2}$,
then we have
$$
(E_{3})^{2}E_{2}-(q^{2}+q^{-2})E_{3}E_{2}E_{3}+E_{2}(E_{3})^{2}.
$$
On the other hand,
we need to obtain the relation between $e^{5}$ and $E_{2}$.
$$
\left.
\begin{array}{l}
(m^{+})^{5}_{6}=-(q^{2}-q^{-2})E_{2}(m^{+})^{6}_{6},\\
e^{5}(m^{+})^{5}_{6}=q^{\frac{2}{3}}(m^{+})^{5}_{6}e^{5}+(q-q^{-1})q^{-\frac{4}{3}}(m^{+})^{6}_{6}e^{4},\\
e^{4}(m^{+})^{6}_{6}=q^{-\frac{4}{3}}(m^{+})^{6}_{6}e^{4}.
\end{array}
\right\}
\Longrightarrow
e^{4}=-(q+q^{-1})(E_{2}e^{5}-e^{5}E_{2}).
$$
Then the relation between $e^{5}$ and $E_{2}$ must be deduced from the relation between $e^{4}$ and $E_{2}$,
which is given by the following cross relations
$$
\left.
\begin{array}{l}
(m^{+})^{5}_{6}=-(q^{2}{-}q^{-2})E_{2}(m^{+})^{6}_{6},\\
e^{4}(m^{+})^{5}_{6}=q^{\frac{2}{3}}(m^{+})^{5}_{6}e^{4},\\
e^{4}(m^{+})^{6}_{6}=q^{-\frac{4}{3}}(m^{+})^{6}_{6}e^{4}.
\end{array}
\right\}
\Longrightarrow
e^{4}E_{2}=q^{2}E_{2}e^{4}
\Longrightarrow
(E_{2}e^{5}-e^{5}E_{2})E_{2}=q^{2}E_{2}(E_{2}e^{5}-e^{5}E_{2}).
$$

Combining with $e^{5}=q^{-\frac{4}{3}}(q^{-2}E_{2}e^{6}-e^{6}E_{2})$,
we obtain
$
(E_{2})^{3}e^{6}-(q^{2}+1+q^{-2})(E_{2})^{2}e^{6}E_{2}+(q^{2}+1+q^{-2})E_{2}e^{6}(E_{2})^{2}-e^{6}(E_{2})^{3}=0,
$
namely,
$$
(E_{2})^{3}E_{3}-
\left[
\begin{array}{c}
3\\
1
\end{array}
\right]
_{q}
(E_{2})^{2}E_{3}E_{2}
+
\left[
\begin{array}{c}
3\\
2
\end{array}
\right]
_{q}
E_{2}E_{3}(E_{2})^{2}-E_{3}(E_{2})^{3}=0.
$$
From these relations,
it follows that the resulting quantum group is $U_{q}({\mathfrak {sp}}_{6})$.
\end{proof}

\subsection{Type-crossing construction from type $A_3$ to type $D_4$}
In the following example, we can choose the different braided groups generated by a $6$-dimensional $U_{q}({\mathfrak {sl}}_{4})$-module to give $U_{q}({\mathfrak {so}}_{8})$ based on the node diagram of $U_{q}({\mathfrak {sl}}_{4})$.
\begin{example}
There is a $6$-dimensional irreducible representation $T_{V}$ of $U_{q}({\mathfrak {sl}}_{4})$,
given by
\begin{gather*}
E_{1}
\left(
\begin{array}{c}
x_{2}\\
x_{3}
\end{array}
\right)
=
\left(
\begin{array}{c}
x_{4}\\
x_{5}
\end{array}
\right)
, \quad
E_{2}
\left(
\begin{array}{c}
x_{1}\\
x_{5}
\end{array}
\right)
=
\left(
\begin{array}{c}
x_{2}\\
x_{6}
\end{array}
\right)
,\quad
E_{3}
\left(
\begin{array}{c}
x_{2}\\
x_{4}
\end{array}
\right)
=
\left(
\begin{array}{c}
x_{3}\\
x_{5}
\end{array}
\right);\\
F_{1}
\left(
\begin{array}{c}
x_{4}\\
x_{5}
\end{array}
\right)
=
\left(
\begin{array}{c}
x_{2}\\
x_{3}
\end{array}
\right)
,\quad
F_{2}
\left(
\begin{array}{c}
x_{2}\\
x_{6}
\end{array}
\right)
=
\left(
\begin{array}{c}
x_{1}\\
x_{5}
\end{array}
\right),
\quad
F_{3}
\left(
\begin{array}{c}
x_{3}\\
x_{5}
\end{array}
\right)
=
\left(
\begin{array}{c}
x_{2}\\
x_{4}
\end{array}
\right).
\end{gather*}
Here $\{x_{i}\mid 1\leq i\leq 6\}$ is a basis of $V$ with corresponding weights
$-2\lambda_{1}+\alpha_{1},
-2\lambda_{1}+\alpha_{1}+\alpha_{2},
-2\lambda_{1}+\alpha_{1}+\alpha_{2}+\alpha_{3},
-2\lambda_{1}+2\alpha_{1}+\alpha_{2},
-2\lambda_{1}+2\alpha_{1}+\alpha_{2}+\alpha_{3},
-2\lambda_{1}+2\alpha_{1}+2\alpha_{2}+\alpha_{3}$,
respectively.
Then we get another $36\times 36$ matrix $R_{VV}$,
and the $PR_{VV}$ matrix obeys the minimal polynomial
$$(PR_{VV}+q^{-1}I)(PR_{VV}-q^{-1}I)(PR_{VV}-qI)=0.$$

Setting $R=qR_{VV}$, $R'=RPR-(q^{2}+1)R+(q^{2}+1)P$,
then we have
$$(PR+I)(PR^{\prime}-I)=0,$$
and $\lambda=q^{-1}$.
Identify $e^{6},f_{6},(m^{+})^{6}_{6}c^{-1}$ with the additional simple root vectors $E_{4},\, F_{4}$ and the group-like element $K_{4}$.
Then the resulting quantum group $U(V^{\vee}(R^{\prime}, R_{21}^{-1}),\widetilde{U_{q}^{ext}({\mathfrak {sl}}_{4})}, V(R^{\prime},R))$
is the quantum group $U_{q}({\mathfrak {so}}_{8})$ with $K_{i}^{\frac{1}{2}}, \ 1\leq i\leq3$ adjoined.
\end{example}
\begin{proof}
The entries in matrix $m^{\pm}$ we need are listed by the following equalities
$$
\left\{
\begin{array}{l}
(m^{+})^{1}_{2}=-(q-q^{-1})E_{2}K_{1}^{\frac{1}{2}}K_{3}^{\frac{1}{2}}, \quad
(m^{+})^{2}_{2}=K_{1}^{\frac{1}{2}}K_{3}^{\frac{1}{2}},\\
(m^{+})^{2}_{3}=-(q-q^{-1})E_{3}K_{1}^{\frac{1}{2}}K_{3}^{-\frac{1}{2}}, \quad
(m^{+})^{3}_{3}=K_{1}^{\frac{1}{2}}K_{3}^{-\frac{1}{2}},\\
(m^{+})^{2}_{4}=-(q-q^{-1})E_{1}K_{1}^{-\frac{1}{2}}K_{3}^{\frac{1}{2}}, \quad
(m^{+})^{4}_{4}=K_{1}^{-\frac{1}{2}}K_{3}^{\frac{1}{2}}, \quad
(m^{+})^{6}_{6}=K_{1}^{-\frac{1}{2}}K_{2}^{-1}K_{3}^{-\frac{1}{2}},\\
(m^{-})^{2}_{1}=q(q-q^{-1})K_{1}^{-\frac{1}{2}}K_{3}^{-\frac{1}{2}}F_{2}, \quad
(m^{-})^{2}_{2}=K_{1}^{-\frac{1}{2}}K_{3}^{-\frac{1}{2}},\\
(m^{-})^{3}_{2}=q(q-q^{-1})K_{1}^{-\frac{1}{2}}K_{3}^{\frac{1}{2}}F_{3}, \quad
(m^{-})^{3}_{3}=K_{1}^{-\frac{1}{2}}K_{3}^{\frac{1}{2}},\\
(m^{-})^{5}_{3}=q(q-q^{-1})K_{1}^{\frac{1}{2}}K_{3}^{\frac{1}{2}}F_{1}, \quad
(m^{-})^{5}_{5}=K_{1}^{\frac{1}{2}}K_{3}^{\frac{1}{2}}.
\end{array}
\right.
$$

The cross relations can be easily obtained as above,
so we only describe the $q$-Serre relations of the positive part.
$E_{1}, E_{3}$ belong to $(m^{+})^{2}_{4}, \ (m^{+})^{2}_{3}$, respectively,
so we have
$$
\left.
\begin{array}{l}
(m^{+})^{2}_{4}=-(q-q^{-1})E_{1}(m^{+})^{4}_{4},\\
e^{6}(m^{+})^{2}_{4}=\lambda R^{2}_{2}{}^{6}_{6}(m^{+})^{2}_{4}e^{6}=(m^{+})^{2}_{4}e^{6},\\
e^{6}(m^{+})^{4}_{4}=(m^{+})^{4}_{4}e^{6}.
\end{array}
\right\}
\Longrightarrow
e^{6}E_{1}=E_{1}e^{6},
\Longrightarrow
E_{4}E_{1}=E_{1}E_{4}.
$$
$$
\left.
\begin{array}{l}
(m^{+})^{2}_{3}=-(q-q^{-1})E_{3}(m^{+})^{3}_{3},\\
e^{6}(m^{+})^{2}_{3}=\lambda R^{2}_{2}{}^{6}_{6}(m^{+})^{2}_{3}e^{6}=(m^{+})^{2}_{3}e^{6},\\
e^{6}(m^{+})^{3}_{3}=(m^{+})^{3}_{3}e^{6}.
\end{array}
\right\}
\Longrightarrow
e^{6}E_{3}=E_{3}e^{6},
\Longrightarrow
E_{4}E_{3}=E_{3}E_{4}.
$$

$E_{2}$ belongs to $(m^{+})^{1}_{2},$
then we obtain
$$
\left.
\begin{array}{l}
(m^{+})^{1}_{2}=-(q-q^{-1})E_{2}(m^{+})^{2}_{2},\\
e^{6}(m^{+})^{1}_{2}
=q^{-1}(m^{+})^{1}_{2}e^{6}+q^{-1}(q-q^{-1})(m^{+})^{2}_{2}e^{5},\\
e^{6}(m^{+})^{2}_{2}=(m^{+})^{2}_{2}e^{6}, \
e^{5}(m^{+})^{2}_{2}=q^{-1}(m^{+})^{2}_{2}e^{5}.
\end{array}
\right\}
\Longrightarrow
e^{5}=q^{-1}E_{2}e^{6}-e^{6}E_{2}.
$$

Combining with
$
e^{6}e^{5}=qe^{5}e^{6},
$
which is obtained by
$e^{i}e^{j}=R'{}^{ji}_{ab}e^{a}e^{b}$ and
$R'{}^{5}_{5}{}^{6}_{6}=-2q,
R'{}^{5}_{6}{}^{6}_{5}=3$,
we have
$$
E_{2}(E_{4})^{2}-(q+q^{-1})E_{4}E_{2}E_{4}+(E_{4})^{2}E_{2}=0.
$$
On the other hand,
combining with $e^{5}E_{2}=qE_{2}e^{5}$ deduced from $e^{5}(m^{+})^{1}_{2}=\lambda R^{1}_{1}{}^{5}_{5}(m^{+})^{1}_{2}e^{5}$,
we obtain
$
E_{4}(E_{2})^{2}-(q+q^{-1})E_{2}E_{4}E_{2}+(E_{2})^{2}E_{4}=0.
$
\end{proof}
\begin{remark}
Observing the constructions of $C_{3}$ and $D_{4}$,
we claim that $C_{n+1}$ and $D_{n+1}$ can be constructed directly from the node diagram $A_n$.
The pairs of braided groups or the $R$-matrices data $R, R^{\prime}$ we should choose will be obtained by
the `symmetric square' and the second exterior power of the vector representation of $A_{n}$.
The verification of the claim will rely on some skills,
full details of it will be developed in a sequel.
\end{remark}
%\end{format}%正文格式结束
%%%%%%%%%%%%%%%%%%%%%%%%%%%%%
\bigskip
%\newpage

\end{document}